\documentclass{article}
\usepackage[german, english]{babel}
\usepackage{amssymb,amsmath}
\usepackage{amsmath, amsthm}
\usepackage{amsfonts,epsfig}
\usepackage{hyperref}

\usepackage[margin=1.1in]{geometry}

\newtheoremstyle{theorem}%name
  {10pt}		  % space above
  {10pt}  % space below
  {\sl}  % bofy font
  {\parindent}     % ident - empty=no indent,  \parindent= paragraph indent
  {\bf}  % thm head font
  {. }    % punctuation after thm head
  { }    % space after thm head: `` ``=normal \newline=linebreak
  {}     % thm head specification
\theoremstyle{theorem}
\newtheorem{theorem}{Theorem}
\newtheorem{corollary}{Corollary}
\newtheorem{proposition}{Proposition}

\newtheorem{definition}{Definition}

\newcommand{\naturals}{\mathbb{N}}

\newcommand{\reals}{\mathbb{R}}
\newcommand{\complex}{\mathbb{C}}

\newcommand{\A}{{\rm \mathbf{A}}}

\newcommand{\D}{{\rm \mathbf{D}}}

\newcommand{\I}{{\rm \mathbf{I}}}

\begin{document}
\bibliographystyle{plain}

\title{\textbf{\Large{Complete bases of canonical polynomials and the\\full  recursive formulation of the spectral-tau method
			}}
}
\author{Alexandros G. Paraskevopoulos\\
\small The Center for Research and Applications of Nonlinear Systems (CRANS)\\
\small Department of Mathematics, Division of Applied Analysis,\\
\small University of Patras, 26500 Patra, Greece}

\maketitle

\begin{abstract} A decomposition of a higher order linear differential operator with polynomial coefficients into a direct sum of two factor operators is obtained. This leads to a lower echelon matrix representation for operators of the above mentioned type arising from a finite sequence of elementary operations. As a consequence, a complete basis of canonical polynomials is generated via a modified Ortiz' recurrence formula. The approach presented in this paper enables us to curry out the full and self-operative recursive implementation of the spectral-tau method.  
These results can be applied to a variety of operator equations associated with infinite matrices in lower row echelon form, having implications for a broad family of discrete and continuous numerical methods, interpreted as special realizations of the tau method under specific perturbation terms.
\end{abstract}

\section{Introduction} \label{sec:Introduction}
Canonical polynomials (CPs) were introduced by Lanczos in~\cite{La:An} along with the tau method and used by Ortiz in~\cite{Or:tau} for the recursive formulation of this method. 
%To produce a more accurate approximation...we need not repeat the whole computation, since in the next step of the computation we shall use the canonical polynomials already determined". 
Ortiz offered, in the reference cited above, a formal definition of the CPs, associated with higher order linear differential operators with polynomial coefficients, along with a recursive formula generating sequences of CPs. The recursive formula has been followed by an algorithm for the generation of some undetermined CPs, due to the appearance of singularities in the recursive construction of a  sequence of CPs. Such singularities made it difficult to program the recursive formulation of the tau method via CPs. An operational formulation of the tau method was proposed by Ortiz and Samara in~\cite{Sa:op}, as an alternative to the recursive formulation of this method.  It made it possible to program the tau method loosing, however, the recursive advantages regarding CPs. 
Singular CPs were further analysed and classified by Bunchaft in~\cite{Bu:ext} as either primary-singular or derived CPs. Moreover an extension of Ortiz' recursive formula is accomplished there so as to recover primary-singular CPs. 

Ortiz' representation theorems for the recursive formulation of the tau method have been generalized in~\cite{Pa:reap} leading to a recursive approach for the exact solution of linear equations on infinite dimensional vector spaces. Given an arbitrary well-ordered basis $\mbox{\boldmath $\mathrm{b}$}=(\mbox{\boldmath $\mathrm{b}$}_{i})_{i \in I}$ of the codomain space of a linear mapping $f$, the axiom of choice in connection with Noether bases of subspaces of a vector space would serve to generalize the notion of sequences of CPs, therein identified as \emph{families of Ortiz canonical vectors} (\emph{OC}-vectors for short) associated with $f$ relative to $\mbox{\boldmath $\mathrm{b}$}$. In analogy with the case of differential operators, a family of \emph{OC}-vectors is also generated by means of a recursive formula with the aid of a \emph{standard family} associated with $f$, relative to $\mbox{\boldmath $\mathrm{b}$}$. The generalized recursive formula has the advantage of yielding a \emph{complete family} of \emph{OC}-vectors in the sense that the so-derived family can be extended to a basis of the domain space of the linear mapping by means of elements of a basis of its null space, termed Ortiz canonical basis (\emph{OC}-basis). It turns out that a complete family of \emph{OC}-vectors is a basis for the representation of a solution $\mbox{\boldmath ${x}$}$ of a linear equation 
\begin{equation} \label{eq-abstract}
	                f(\mbox{\boldmath ${x}$})=\mbox{\boldmath ${y}$},
\end{equation}
such that the coefficients of every right hand side vector $\mbox{\boldmath ${y}$}$ in the range of $f$, relative to $\mbox{\boldmath $\mathrm{b}$}$, remain invariant in the expression of $\mbox{\boldmath ${x}$}$ omitting the coefficients indexed by the set $S$ of \emph{inaccessible index}. In analogy with an \emph{OC}-basis, a standard family supplemented by the elements of a basis of the null space of $f$ yields a \emph{standard basis} associated with $f$ relative to $\mbox{\boldmath $\mathrm{b}$}$. Moreover, a natural mutual association between standard bases and infinite matrices in lower row echelon form (LREF) is established there, which enabled us to identify standard bases as bases of the domain space generating matrices in LREF. As the tau approximations are the exact polynomial solutions of a perturbed equation, the problem of their construction in terms of \emph{OC}-bases is completely reduced to the problem of transforming the infinite matrix representation of a linear operator into an LREF. Accordingly infinite matrices in LREF provide a general and unified scheme for the recursive construction of \emph{OC}-bases associated with linear operators. Some concrete examples from differerent areas of applications are also given in~\cite{Pa:reap}. 

Let $D= \sum_{i=0}^{\nu} p_{i}(x) \frac{d^{i}}{dx^{i}}$ be a differential operator of order $\nu$ with polynomial coefficients restricted to the space ${\mathcal{P}}({\reals})$ of real polynomials. In this paper we show that it is always possible to decompose $D$ into a direct sum of two factor operators, say $D=D_{1}\oplus D_{2}$, which possesses the following properties: \\
(i) $D_{1}$ is defined on a finite dimensional subspace of ${\mathcal{P}}({\reals})$ and $D_{2}$ is an isomorphism between infinite dimensional subspaces of ${\mathcal{P}}({\reals})$.\\ 
(ii) The domain space of $D_{1}$ contains the polynomial Kernel of $D$. \\
(iii) The codomain space of $D_{1}$ contains the space of residual polynomials, that is an algebraic complement of the polynomial range of $D$ spanned by $\{x^{s}\}$, where $s$ ranges over the set $S$ of inaccessible index. 

Such a decomposition of $D$ arises from the determination of a nonnegative integer $N$, that is the greatest root of a polynomial equation in which the unknown varies over the set $\naturals$ (the set of natural numbers including zero). If such $N$ does not exist, then $D$ is identical to $D_{2}$ and the matrix representation of $D$ is in LREF relative to the usual polynomial basis ${\mbox{\boldmath $\mathrm{x}$}}=\{x^{n},\;\;n\in \naturals\}$. Otherwise, the LREF of the matrix representation of $D$ is constructed after a finite sequence of row elementary operations applied to the matrix associated with $D_{1}$. In both cases we conveniently derive a matrix representation of $D$ in LREF coupled with a standard polynomial basis including a basis of the polynomial kernel of $D$. In Ortiz' recurrence for generating canonical polynomials, the terms of the polynomial basis $\mbox{\boldmath $\mathrm{x}$}$ are replaced by corresponding terms of a standard polynomial basis and the coefficients, which take place in this formula, match the nonzero entries of the LREF representation of $D$. The outcome of this formula is a basis of CPs, that is a \emph{OC}-basis of the entire space of polynomials, partitioned into a complete sequence of CPs and a basis of the polynomial kernel of $D$. In view of Bunchaft's classification of CPs the modified recursive formula derived in this paper results in a sequence consisting of all primary-generic CPs along with a self-operative selection of primary-singular and derived CPs so as to form a complete sequence of such polynomials. As a consequence a simple criterion is introduced, which makes it possible to establish whether or not every complete sequence of CPs associated with a given $D$ contains at least one derived-singular CP.

The two-factor decomposition of $D$, mentioned above, enables us to show that the height $h$ of $D$ is intrinsically related to a more general quantity, called the index of the operator, which characterizes operators of finite Kernel index and deficiency. It turns out that the number of tau parameters must be the constant integer $h+\nu$ for any order of approximation $n>N$.

Taking advantage of programming Languages which support symbolic computations, the recursive procedure discussed in this paper directly leads to a computer program for the full recursive implementation of the tau method, which additionally preserves the above-mentioned benefits regarding CPs (see Examples 1,2). An extension of the approach presented in this paper to cover the case of systems of differential equations determined by matrix differential operators with polynomial coefficients is presented in \cite{Pa:sy}.

In a series of papers~\cite{El:un, El:co, El:Ga, El:an}  El-Daou et al. simulated a variety of discrete and continuous numerical methods, such as finite difference methods, collocation schemes, spectral techniques and the tau method, through a systematic use of the recursive formulation of the latter as the main analytical tool. In this context our results are connected with a fairly large family of numerical methods aimed at the full recursive construction of their approximate solution in terms of bases of canonical polynomials.

\section{Ortiz Canonical Bases and Infinite Matrices in Echelon Form}
\label{sec:TheRecursiveGenerationOfOrtizCanonicalBasesViaInfiniteMatricesInEchelonForm}

In this Section we collect some results arising out of the recursive approach to the solution of abstract linear equations demonstrated by the author in \cite{Pa:reap}. 

In all that follows $X, Y$ will stand for vector spaces of arbitrary (finite or infinite) Hamel dimension over the same scalar field $\mathcal{F}$ (the field $\complex$ of complex numbers or any subfield of $\complex$). The space of linear mappings (or homomorphisms) from $X$ into $Y$ over $\mathcal{F}$ will be denoted by $\mbox{$\mathrm{Hom}$}_{\mathcal{F}}(X,Y)$. Also $\mbox{$\mathrm{Im}$}(f)$, $\mbox{$\mathrm{Ker}$}(f)$ will stand for the \emph{range} (or \emph{image}) and the \emph{null space} (or \emph{kernel}) of $f\in \mbox{$\mathrm{Hom}$}_{\mathcal{F}}(X,Y)$ respectively. 
Let $\mbox{\boldmath $\mathrm{s}$}=(\mbox{\boldmath $\mathrm{s}$}_{k})_{k \in K}$, $\mbox{\boldmath $\mathrm{b}$}=(\mbox{\boldmath $\mathrm{b}$}_{i})_{i \in I}$ be Hamel bases of $X$, $Y$, respectively, well-ordered by $(K,\prec_{K})$, $(I,\prec_{I})$. Let also $\A=(\alpha_{ki})_{(k,i) \in K \times I}$ be the row finite matrix representation of $f\in \mbox{$\mathrm{Hom}$}_{\mathcal{F}}(X,Y)$, relative to $(\mbox{\boldmath $\mathrm{s}$},\: \mbox{\boldmath $\mathrm{b}$})$ determined by $f(\mbox{\boldmath $\mathrm{s}$}_{k})=\sum_{i\in I} \alpha_{ki}\mbox{\boldmath $\mathrm{b}$}_{i}$ for $k\in K$. A row finite matrix is uniquely determined by $f$ for an arbitrary but fixed pair of bases $(\mbox{\boldmath $\mathrm{s}$},\: \mbox{\boldmath $\mathrm{b}$})$ and vice versa (see \cite{Ja:Ab}). The row of index $i$ of $\A$ is denoted as $\A_i$.
The set $W$ will stand for the indexing set of zero rows of $\A$, that is $W=\{k\in K: f(\mbox{\boldmath $\mathrm{s}$}_{k})=0\}$ and the set $J$ for the set complement of $W$, that is $J=K\verb+\+W$. Let $\mbox{\boldmath $\mathrm{y}$}=\sum_{i  \in I} \alpha_{i}\mbox{\boldmath $\mathrm{b}$}_{i}$ be an element of $Y$ and $supp(\mbox{\boldmath $\mathrm{y}$})=\{i\in I: \alpha_{i}\not=0\}$ be the finite support of $\mbox{\boldmath $\mathrm{y}$}$. The set $supp(\mbox{\boldmath $\mathrm{y}$})$ has a unique greatest element, relative to $\prec_{I}$, and we shall refer to it as $gsup(\mbox{\boldmath $\mathrm{y}$})$, relative to \mbox{\boldmath $\mathrm{b}$}.
Let us further call $\sigma_{j}=gsup(f(\mbox{\boldmath $\mathrm{s}$}_{j}))$ for $j\in J$. Then we can write $f(\mbox{\boldmath $\mathrm{s}$}_{j})=\sum_{i\preceq_{I} \sigma_{j}} \alpha_{ji}\mbox{\boldmath $\mathrm{b}$}_{i}$ for all $j\in J$.

\subsection{Two Special Bases of Subspaces of a Vector Space}
\label{sec:TwoSpecialBasesOfSubspacesOfAVectorSpace}
Let $G$ be a subspace of $Y$. The set of inaccessible index associated with $G$, relative to an ordered basis $\mbox{\boldmath $\mathrm{b}$}=(\mbox{\boldmath $\mathrm{b}$}_{i})_{i \in I}$ of $Y$, is defined to be the set $S=\{s\in I: \:s\not=gsup(\mbox{\boldmath $\mathrm{g}$})\;\forall\mbox{\boldmath $\mathrm{g}$}\in G\}$. The set $S$ would serve to introduce two special types of ordered bases of $G$.

In all that follows, $S^{\prime}$ will stand for the set complement of $S$ with respect to $I$, that is $S^{\prime}=I\verb+\+ S$. Let us call ${\mathcal{M}}_{i}=\{\mbox{\boldmath $\mathrm{g}$}\in G:\;gsup(\mbox{\boldmath $\mathrm{g}$})=i, \mbox{$\mathrm{relative\; to}$}\  \mbox{\boldmath $\mathrm{b}$}\}$ for all $i\in S^{\prime}$. Let also $G\not=\{0\}$. Formally the family ${\mathcal{M}}=({\mathcal{M}}_{i})_{i\in S^{\prime}}$ is a covering of $G\verb+\+\{0\}$ consisting of pairwise disjoint and nonempty subsets of $G\verb+\+\{0\}$. As ${\mathcal{M}}_{i}\not=\emptyset$ for all $i\in S^{\prime}$, the axiom of choice implies the existence of a choice mapping ${\mathcal{C}}:{\mathcal{M}}\mapsto G$ such that ${\mathcal{C}}({\mathcal{M}}_{i})=\mbox{\boldmath $\mathrm{g}$}_{i}\in {\mathcal{M}}_{i}$ for all $i\in S^{\prime}$. As the terms of ${\mathcal{M}}$ are pairwise disjoint sets, it follows that ${\mathcal{C}}$ is injective and so the family $\mbox{\boldmath $\mathrm{g}$}=(\mbox{\boldmath $\mathrm{g}$}_{i})_{i\in S^{\prime}}$ in $G$ consists of non-repeating terms. Moreover as the terms of $\mbox{\boldmath $\mathrm{g}$}$ satisfy the defining relation 
\begin{equation}\label{eq-prestand} 
gsup(\mbox{\boldmath $\mathrm{g}$}_{i})=i 
\end{equation}
for all $i\in S^{\prime}$, relative to $\mbox{\boldmath $\mathrm{b}$}$, they can be written in the form
\begin{equation}\label{eq-stand1} 
	\mbox{\boldmath $\mathrm{g}$}_{i}=\sum_{k\preceq i} \lambda_{ik}\mbox{\boldmath $\mathrm{b}$}_{k}
\end{equation}
for $\lambda_{ik}\in {\mathcal{F}}$ with $\lambda_{ii}\not=0$. It turns out that the family $(\mbox{\boldmath $\mathrm{g}$}_{i})_{i\in S^{\prime}}\cup (\mbox{\boldmath $\mathrm{b}$}_{s})_{s\in S}$ is a basis of $Y$ and the family ${\mathcal{G}}=(\mbox{\boldmath $\mathrm{g}$}_{i})_{i\in S^{\prime}}$ is a basis of $G$ (see \cite[Lemma 1]{Pa:reap}). We shall referred to this basis as \emph{full-index basis} of $Y$ (resp. $G$) relative to $\mbox{\boldmath $\mathrm{b}$}$. Moreover, the space ${\mathcal{R}}_{S}$ spanned by $(\mbox{\boldmath $\mathrm{b}$}_{s})_{s\in S}$ is a complementary space of $G$ termed the \emph{residual space} of $G$, relative to $\mbox{\boldmath $\mathrm{b}$}$  and we write formally for it
\begin{equation}\label{eq-fres} 
	Y=G\oplus \mathcal{R_{S}}. 
\end{equation}
According to (\ref{eq-fres}), for every $i\in S^{\prime}$, $\mbox{\boldmath $\mathrm{b}$}_{i}$ can be expressed uniquely in the form
\begin{equation} \label{eq-Noe}
	\mbox{\boldmath $\mathrm{b}$}_{i}=\mbox{\boldmath $\mathrm{n}$}_{i}-\mbox{\boldmath $\mathrm{r}$}_{i},
\end{equation}
where $\mbox{\boldmath $\mathrm{n}$}_{i}\in G$ with $\mbox{\boldmath $\mathrm{n}$}_{i}\not=0$ and $\mbox{\boldmath $\mathrm{r}$}_{i}\in \mathcal{R_{S}}$. The family $\mbox{\boldmath $\mathrm{r}$}=(\mbox{\boldmath $\mathrm{r}$}_{i})_{i\in S^{\prime}}$ with $\mbox{\boldmath $\mathrm{r}$}_{i}\in\mathcal{R_{S}}$ will be referred to as the family of residual vectors (\emph{R}-vectors for short) associated with $G$ relative to $\mbox{\boldmath $\mathrm{b}$}$. As $\mathcal{R_{S}}$ is a complementary space of $G$, it follows from Noether's Theorem (see \cite{Ja:Ab}) that the family ${\mathcal{N}}=(\mbox{\boldmath $\mathrm{n}$}_{i})_{i\in S^{\prime}}$, generated by (\ref{eq-Noe}), is a basis of $G$, named \emph{Noether basis} of $G$ relative to $\mbox{\boldmath $\mathrm{b}$}$.

\subsection{Standard Bases and Ortiz Canonical Bases}
\label{sec:StandardBasesAndOrtizCanonicalBasesAssociatedWithAbstractLinearMappings}

Let $\mbox{$\mathrm{Hom}$}_{\mathcal{F}}(X,Y)$ be the space of linear mappings (or homomorphisms) from $X$ into $Y$ over the same scalar field $\mathcal{F}$. Let also $f\in \mbox{$\mathrm{Hom}$}_{\mathcal{F}}(X,Y)$. The kernel (or null space) and the range (or image) of $f$ will be denoted by $\mbox{$\mathrm{Ker}$}(f),\;\mbox{$\mathrm{Im}$}(f)$ respectively and $\tilde{f}$ will stand for the canonical isomorphism induced by $f$, that is $\tilde{f}: X/\mbox{$\mathrm{Ker}$}(f)\mapsto\mbox{$\mathrm{Im}$}(f)$, $\tilde{f}([\mbox{\boldmath $\mathrm{x}$}])=f(\mbox{\boldmath $\mathrm{x}$})$. In what follows, we shall apply the results of the preceding paragraph with $G=\mbox{$\mathrm{Im}$}(f)$ and $f\neq 0$. Let $\mbox{\boldmath $\mathrm{b}$}=(\mbox{\boldmath $\mathrm{b}$}_{i})_{i \in I}$ be an ordered basis of $Y$. Since  $\mbox{$\mathrm{Im}$}(f)\neq \{0\}$, it follows that $S$ is a proper subset of $I$ and $S^{\prime}\neq\emptyset$. Accordingly, the full-index basis ${\mathcal{G}}$ and the Noether basis  ${\mathcal{N}}$ of $\mbox{$\mathrm{Im}$}(f)$, relative to $\mbox{\boldmath $\mathrm{b}$}$ are non-empty sets. The inverse mapping of $\tilde{f}$, as being an isomorphism, generates the families of cosets $(\tilde{f}^{-1}(\mbox{\boldmath $\mathrm{g}$}_{i}))_{i\in S^{\prime}}$ and $(\tilde{f}^{-1}(\mbox{\boldmath $\mathrm{n}$}_{i}))_{i\in S^{\prime}}$, which are formally bases of $X/\mbox{$\mathrm{Ker}$}(f)$. The former will be referred to as a \emph{standard basis of cosets} and the latter as a \emph{basis of Lanczos cosets} associated with $f$, relative to $\mbox{\boldmath $\mathrm{b}$}$. Such bases consist of pairwise disjoint and nonempty sets and so, by virtue of the axiom of choice, there exist injective choice mappings, which generate a \emph{standard family of vectors} and a \emph{family of Ortiz canonical vectors} (family of \emph{OC}-vectors) associated with $f$, relative to $\mbox{\boldmath $\mathrm{b}$}$, respectively. It turns out that the so generated families consist of terms indexed by $S^{\prime}$ with none repeated. Moreover such families are linearly independent and span complementary spaces of $\mbox{$\mathrm{Ker}$}(f)$. 

Let $\mbox{\boldmath $\mathrm{s}$}_{i}$ be the $i$-th term of a standard family, say $\mbox{\boldmath $\mathrm{s}$}$, associated with $f$, relative to $\mbox{\boldmath $\mathrm{b}$}$. In view of (\ref{eq-prestand}), we have
\begin{equation}\label{eq-standfa1} 
gsup\;f(\mbox{\boldmath $\mathrm{s}$}_{i})=i, 
\end{equation}
whence
\begin{equation}\label{eq-standfa}
	f(\mbox{\boldmath $\mathrm{s}$}_{i})=\sum_{k\preceq i} \lambda_{ik}\mbox{\boldmath $\mathrm{b}$}_{k},
\end{equation}
for some $\lambda_{ik}\in {\mathcal{F}}$ with $\lambda_{ii}\not=0$. Let $\mbox{\boldmath $\mathrm{q}$}_{i}$  be the $i$-th term of a family of \emph{OC}-vectors, say $\mbox{\boldmath $\mathrm{q}$}$, associated with $f$, relative to $\mbox{\boldmath $\mathrm{b}$}$. It follows that $f(\mbox{\boldmath $\mathrm{q}$}_{i})=\mbox{\boldmath $\mathrm{n}$}_{i}$, or in view of (\ref{eq-Noe}),
\begin{equation}\label{eq-can2}
	f(\mbox{\boldmath $\mathrm{q}$}_{i})=\mbox{\boldmath $\mathrm{b}$}_{i}+\mbox{\boldmath $\mathrm{r}$}_{i}.
\end{equation}
As $\tilde{f}^{-1}$ is injective, the uniqueness of Noether bases defined by (\ref{eq-Noe}) entails the following result. 
\begin{theorem}\label{theo-uniq} A basis of Lanczos cosets ${\mathcal{L}}=({\mathcal{L}}_{i})_{i\in S^{\prime}}$ and a family of residual vectors $\mbox{\boldmath $\mathrm{r}$}=(\mbox{\boldmath $\mathrm{r}$}_{i})_{i\in S^{\prime}}$ are uniquely associated with $f\in \mbox{$\mathrm{Hom}$}_{\mathcal{F}}(X,Y)$ for a fixed but arbitrary ordered basis $\mbox{\boldmath $\mathrm{b}$}$ of $Y$. 
\end{theorem}
If $f$ is not injective, then $\mbox{$\mathrm{Ker}$}(f)\neq \{0\}$ and so each coset ${\mathcal{L}}_{i}$, $i\in S^{\prime}$, consists of multiple \emph{OC}-vectors of the same index satisfying each distinct equation (\ref{eq-can2}). It turns out that, in the general case, there exist multiple families of \emph{OC}-vectors. 
However, by virtue of Theorem \ref{theo-uniq}, we deduce the following property.
\begin{corollary} \label{cor-uniq}
Any two terms of two arbitrary families of OC-vectors associated with $f$, relative to $\mbox{\boldmath $\mathrm{b}$}$, have the same index in $S^{\prime}$ if and only if they differ by an element of $\mbox{$\mathrm{Ker}$}(f)$. 
\end{corollary}
Let $\mbox{\boldmath $\mathrm{u}$}=(\mbox{\boldmath $\mathrm{u}$}_{w})_{w\in W}$ be a basis of $\mbox{$\mathrm{Ker}$}(f)$. As $\mbox{\boldmath $\mathrm{s}$}$ and $\mbox{\boldmath $\mathrm{q}$}$ span complementary subspaces of $\mbox{$\mathrm{Ker}$}(f)$, respectively, it follows that their  extensions $\tilde{\mbox{\boldmath $\mathrm{s}$}}=\mbox{\boldmath $\mathrm{s}$}\cup \mbox{\boldmath $\mathrm{u}$}$ and $\tilde{\mbox{\boldmath $\mathrm{q}$}}=\mbox{\boldmath $\mathrm{q}$}\cup \mbox{\boldmath $\mathrm{u}$}$ are bases of $X$. We shall refer to these bases as a standard basis and a \emph{OC}-basis associated with $f$, relative to $\mbox{\boldmath $\mathrm{b}$}$, respectively. The indexing set of $\tilde{\mbox{\boldmath $\mathrm{s}$}}$ and $\tilde{\mbox{\boldmath $\mathrm{q}$}}$ is the disjoint union of the sets $W,\:S^{\prime}$, which will be denoted by $W\coprod S^{\prime}$. Let $\prec_{W}$ be a well-ordering on $W$. Then the families $\tilde{\mbox{\boldmath $\mathrm{s}$}}$ and $\tilde{\mbox{\boldmath $\mathrm{q}$}}$ will be considered as ordered bases equipped with the conjoint well-ordering $^{W}\!\lhd=\prec_{W} \cup \prec_{S^{\prime}}\cup (W\times S^{\prime})$. According to this ordering the elements of $W$ are predecessors of the elements of $S^{\prime}$.
\subsection{Existence and Solution  Representation}
\label{sec:ExistenceAndRepresentationOfSolutionsOfLinearEquations}
Let ${\mbox{\boldmath ${y}$}}=\sum_{i\in I} \lambda_{i}\mbox{\boldmath $\mathrm{b}$}_{i}$ be an arbitrary element of $Y$ and $\mbox{\boldmath $\mathrm{r}$}=(\mbox{\boldmath $\mathrm{r}$}_{i})_{i \in S^{\prime}}$ be the family of \emph{R}-vectors associated with $f$ relative to $\mbox{\boldmath $\mathrm{b}$}$. Then $\mbox{\boldmath ${y}$}$ belongs to the range of $f$ if and only if 
\begin{equation}\label{eq-exis1}
	\sum_{s\in S}\lambda_{s}\mbox{\boldmath $\mathrm{b}$}_{s}=\sum_{i\in S^{\prime}}\lambda_{i}\mbox{\boldmath $\mathrm{r}$}_{i}
\end{equation}
(see \cite[Theorem 7]{Pa:reap}). Let us write $\mbox{\boldmath $\mathrm{r}$}_{i}=\sum_{s\in S}r_{is}\mbox{\boldmath $\mathrm{b}$}_{s}$. In view of (\ref{eq-exis1}) we have 
\[\begin{array}{rllr}
	\sum_{s\in S}\lambda_{s}\mbox{\boldmath $\mathrm{b}$}_{s}\vspace{0.06in}&=&\sum_{i\in S^{\prime}}\sum_{s\in S} \lambda_{i}r_{is}\mbox{\boldmath $\mathrm{b}$}_{s}\\
\vspace{0.06in}&=&\sum_{s\in S}(\sum_{i\in S^{\prime}}\lambda_{i}r_{is})\mbox{\boldmath $\mathrm{b}$}_{s}.
\end{array}\]
By combining coefficients we conclude that $\mbox{\boldmath ${y}$}$ belongs to the range of $f$ if and only if the following relations hold
\begin{equation}\label{eq-parm}
	\lambda_{s}=\sum_{i\in S^{\prime}}\lambda_{i}r_{is},
\end{equation}
for all $s\in S$.	According to (\ref{eq-parm}) every vector in the range of $f$ must be in the form
\begin{equation}\label{eq-rhs}
\mbox{\boldmath ${y}$}=\sum_{i\in S^{\prime}} \lambda_{i}\mbox{\boldmath $\mathrm{b}$}_{i}+\sum_{s\in S} \sum_{i\in S^{\prime}} \lambda_{i}r_{is}\mbox{\boldmath $\mathrm{b}$}_{s},
\end{equation}
where $\lambda_{i}$ are free scalars in ${\mathcal{F}}$ for all $i\in S^{\prime}$. Let $\mbox{\boldmath ${y}$}$ be as in (\ref{eq-rhs}), let also $\mbox{\boldmath $\mathrm{q}$}=(\mbox{\boldmath $\mathrm{q}$}_{i})_{i \in S^{\prime}}$ be a family of \emph{OC}-vectors associated with $f$ relative to $\mbox{\boldmath $\mathrm{b}$}$ and $\mbox{\boldmath $\mathrm{u}$}=(\mbox{\boldmath $\mathrm{u}$}_{w})_{w\in W}$ be a basis of $\mbox{$\mathrm{Ker}$}(f)$. Then the general solution $\mbox{\boldmath ${x}$}$ of (\ref{eq-abstract}) expressed in terms of an \emph{OC}-basis yields the same coefficients $\lambda_{i}$ for all $i\in S^{\prime}$ as the corresponding coefficients of (\ref{eq-rhs}) and so $\mbox{\boldmath ${x}$}$ takes the form
\begin{equation}\label{eq-repr}
	\mbox{\boldmath ${x}$}=\sum_{i\in S^{\prime}}\lambda_{i}\mbox{\boldmath $\mathrm{q}$}_{i}+ \sum_{w\in W} C_{w}\mbox{\boldmath $\mathrm{u}$}_{w}
\end{equation}
where $C_{w}$ are free scalars.
\subsection{Lower Row Echelon Forms and  the Generalized Recursive Formula}
\label{sec:MatricesInREFAndRecurrence}
An order monomorphism $\sigma : K\ni k\mapsto \sigma(k)\in I$ between well-ordered sets is a strict order preserving mapping, that is if $k \prec_{K} m$, then $\sigma(k) \prec_{I} \sigma(m)$. An order monomorphism that is also surjective (or onto) is called order isomorphism. Next we define the LREF of a row finite matrix indexed by well-ordered sets.
\begin{definition} \label{def-ech} A row finite matrix $\A=(\alpha_{ki})_{(k,i)\in K \times I}$ is said to be in lower row echelon form relative to $(K,\prec_{K})$, $(I,\prec_{I})$ if it fulfills the following properties\emph{:}\\
\emph{i)} $\alpha_{j\sigma_{j}}=1$, for all $j \in J$. \emph{(}The last nonzero element in every nonzero row is the scalar \emph{1)}.\\
\emph{ii)} If $\alpha_{ki}=0$ for some $k \in K$ and for all $i \in I$, then $\alpha_{mi}=0$ for all $m \prec_{K} k$ and for all $i \in I$.  \emph{(}All the predecessors of a zero row are zero rows \emph{)}.\\
\emph{iii)} The mapping $\sigma : J \ni j \mapsto \sigma_{j}=gsup(\A_{j})\in I$ is an order monomorphism. \emph{(}For every nonzero row the last nonzero scalar of any successive row occurs further to the right\emph{)}.
\end{definition}
A matrix satisfying properties ii) and iii) is said to be in \emph{pre}-LREF. It follows directly from (\ref{eq-standfa}) that the action of $f$ on a standard basis results in a matrix in \emph{pre}-LREF.  Some properties of infinite matrices in \emph{pre}-LREF (or LREF), shown in \cite[Lemma 2]{Pa:reap}, are described in what follows. 
\begin{proposition}\label{lem-prop} Let $\A=(\alpha_{ki})_{(k,i) \in K \times I}$ be the matrix representation of $f$ relative to $(\mbox{\boldmath $\mathrm{s}$},\: \mbox{\boldmath $\mathrm{b}$})$ in pre-\emph{LREF}. Let also  $S$ be the  inaccessible index set of $f$, relative to $\mbox{\boldmath $\mathrm{b}$}$. Then\\
\emph{i)} The set complement $S'$ of $S$ satisfies $S^{\prime}=\mbox{$\mathrm{Im}$}(\sigma)$; that amounts to the same to say that $\sigma: J\ni j \mapsto \sigma_{j}\in S^{\prime}$ is surjective.\\
\emph{ii)} The family $(f(\mbox{\boldmath $\mathrm{s}$}_{j}))_{j \in J}$ is a basis of \emph{Im}$(f)$. \\ 
\emph{iii)} The family $(\mbox{\boldmath $\mathrm{s}$}_{w})_{w \in W}$ is a basis of \emph{Ker}$(f)$. 
\end{proposition}
In the general case, $\mbox{$\mathrm{Im}$}(\sigma)\subset S^{\prime}$ and so $\sigma: J\ni j \mapsto \sigma_{j}\in S^{\prime}$ is not  surjective. However, in the case of matrices in \emph{pre}-LREF, by virtue of Proposition \ref{lem-prop} (i) and Definition \ref{def-ech} (iii), the mapping $\sigma: J\ni j \mapsto \sigma_{j}\in S^{\prime}$ is bijective and so $\sigma$ is an order isomorphism. Let us divide a basis $\mbox{\boldmath $\mathrm{s}$}$ associated with a matrix in \emph{pre}-LREF into two disjoint sets $\mbox{\boldmath $\mathrm{M}$}=\{\mbox{\boldmath $\mathrm{s}$}_{j}:\:j \in J\}$ and $\mbox{\boldmath $\mathrm{u}$}=\{\mbox{\boldmath $\mathrm{s}$}_{w}:\:w \in W\}$. Let also $\mbox{\boldmath $\mathrm{s}$}|_J$ be  the restriction of $\mbox{\boldmath $\mathrm{s}$}$ to $J$. As $\sigma, \mbox{\boldmath $\mathrm{s}$}|_J$ are bijections it follows that there exists a unique bijection $\mbox{\boldmath $\mathrm{m}$}$, which renders the diagram
\begin{equation} \label{eq-comm}
\begin{array}{rccl}
      &  &   \sigma      &                       \\
 &  J   &   \longrightarrow     &  S'    \\
\mbox{\boldmath $\mathrm{s}$}|_J & \downarrow &  \swarrow &\!\!\!\!\!\! \mbox{\boldmath $\mathrm{m}$} \\
      &  \mbox{\boldmath $\mathrm{M}$}  &              &          
\end{array}
\end{equation}
commutative. It turns out that the mapping $\mbox{\boldmath  $\mathrm{m}$}$ satisfies $\mbox{\boldmath $\mathrm{m}$}_{\sigma_{j}}=\mbox{\boldmath $\mathrm{s}$}_{j}$ for all $j\in J$. As $\sigma_{j}=gsup(f(\mbox{\boldmath $\mathrm{s}$}_{j}))=gsup(f(\mbox{\boldmath $\mathrm{m}$}_{\sigma_{j}}))$ for all $j\in J$, it follows from (\ref{eq-standfa1}) that $(\mbox{\boldmath $\mathrm{m}$}_{i})_{i\in S^{\prime}}$ is a standard family associated with $f$ relative to $\mbox{\boldmath $\mathrm{b}$}$. Whereas the terms of $\mbox{\boldmath $\mathrm{s}$}|_J$ have the same values  as the terms of $\mbox{\boldmath $\mathrm{m}$}$, but indexed by a different set, it can be identified with a standard family and the basis $\mbox{\boldmath $\mathrm{s}$}$ in Proposition \ref{lem-prop} can be identified with a standard basis associated with $f$ relative to $\mbox{\boldmath $\mathrm{b}$}$. Moreover the existence of standard bases ensures that every linear mapping on vector spaces of arbitrary dimensions is associated with at least one matrix in LREF for an arbitrary but fixed ordered basis of its codomain space. 

The ``existential approach" to the definition of families of \emph{OC}-vectors, based on the axiom of choice, enabled us to formulate and show the general theorems in \cite{Pa:reap}, however this approach is not efficient for computation. An alternative and constructive approach is introduced in connection with infinite row-finite systems in \cite{Pa:igjea}.

The following recursive formula provides a constructive procedure for choosing the terms of a family of \emph{OC}-vectors from Lanczos' cosets, whenever a matrix representation of $f$ in LREF is available. If the matrix representation $\A=(\alpha_{ki})_{(k,i) \in K \times I}$ of $f\in \mbox{$\mathrm{Hom}$}_{\mathcal{F}}(X,Y)$, relative to $(\mbox{\boldmath $\mathrm{s}$},\: \mbox{\boldmath $\mathrm{b}$})$, is in \emph{pre}-LREF, then the terms of a complete family of \emph{OC}-vectors (see \cite[Corollary 4]{Pa:reap}) are generated by 
\begin{equation} \label{eq-cre1}  \mbox{\boldmath $\mathrm{q}$}_{\sigma_{j}}= 
\alpha^{-1}_{j \sigma_{j}}({\mbox{\boldmath $\mathrm{s}$}}_{j} - \sum_{\scriptsize \begin{array}{c} i \prec \sigma_{j} \\ i \notin S\end{array}} \alpha_{j i} \mbox{\boldmath $\mathrm{q}$}_{i})
\end{equation}
for all $j\in J$.
%coupled with the family of \emph{R}-vectors 
%\begin{equation} \label{eq-cre2} \mbox{\boldmath $\mathrm{r}$}_{\sigma_{j}}
%= \alpha^{-1}_{j \sigma_{j}} ( \sum_{\scriptsize \begin{array}{c} i \prec \sigma_{j} \\ i \in S\end{array}} \alpha_{ji}\mbox{\boldmath $\mathrm{b}$}_{i}-\sum_{\scriptsize \begin{array}{c} i \prec \sigma_{j} \\ i \notin S\end{array}} \alpha_{ji} \mbox{\boldmath $\mathrm{r}$}_{i}) 
%\end{equation}
%for all $j\in J$. 

Let us define $\mbox{\boldmath $\mathrm{q}$}_{w}\stackrel{def}{=}{\mbox{\boldmath $\mathrm{s}$}}_{w}$ for all $w\in W$. As $\{\mbox{\boldmath $\mathrm{s}$}_{w}:\:w \in W\}$ is a basis of Ker($f$), a \emph{OC}-basis is given by $\{\mbox{\boldmath $\mathrm{q}$}_{w},\; w\in W\}\cup \{\mbox{\boldmath $\mathrm{q}$}_{\sigma_{j}},\; j\in J\}$ and the indexing set of the corresponding ordered basis $\tilde{\mbox{\boldmath $\mathrm{q}$}}$ is the disjoint union of $W,S^{\prime}$ endowed with the conjoint well-ordering. Since $\alpha_{w i}=0$ (zero rows of $\A$) we can extend (\ref{eq-cre1}) so as to recover all of the terms $\mbox{\boldmath $\mathrm{q}$}_{w}$, $w\in W$, by defining in (\ref{eq-cre1}) $j=\sigma_{j}=w$ and $\alpha^{-1}_{w w}=1$ for all $w\in W$. It turns out that, in the context of matrix representations of linear mappings in LREF, the recursive relation (\ref{eq-cre1}) generates an \emph{OC}-basis.
\section{A Direct Sum Decomposition of Differential Operators}
\label{sec:ADirectDecompositionOfDifferentialOperators}
Throughout this paper $\mathcal{D}$ will stand for the class of differential operators of order $\nu$ with polynomial coefficients, the elements of which are defined by
\begin{equation}\label{eq-odo}
	D= \sum_{i=0}^{\nu} p_{i}(x) \frac{d^{i}}{dx^{i}}, \;\;\;\;  
	p_{i}(x)= \sum_{j=b_{i}}^{\alpha_{i}} p_{ij}x^{j}.
\end{equation}
Inasmuch as $D\in \mathcal{D}$ is an endomorphism of the polynomial space ${\mathcal{P}}({\reals})$, we shall use the notations $\mbox{$\mathrm{Ker}$}(D)$ and $\mbox{$\mathrm{Im}$}(D)$ for the polynomial kernel and the polynomial range of $D$ respectively. Moreover $\mbox{\boldmath $\mathrm{x}$}$ will stand for the usual polynomial basis $\{x^{n},\;\;n\in \naturals\}$ of ${\mathcal{P}}({\reals})$ ordered by the standard order by magnitude on $\naturals$, and $\deg\: p(x)$ for the degree of $p(x)\in {\mathcal{P}}({\reals})$. Certainly $gsup(p(x))=\deg\: p(x)$ relative to $\mbox{\boldmath $\mathrm{x}$}$.

Taking into account that  $\displaystyle \frac{d^{i}}{dx^{i}}x^{n}=\frac{n!}{(n-i)!}x^{n-i}$, the action of $D$ on $\mbox{\boldmath $\mathrm{x}$}$ results in a generating system of Im$(D)$ the terms of which are given by
\begin{equation} \label{eq-genp1}
D(x^{n})=\sum_{i=0}^{\nu}\sum_{j=b_{i}}^{\alpha_{i}}p_{ij} 
\frac{n!}{(n-i)!} x^{n-i+j}.
\end{equation}
The integers $h=max\{\alpha_{i}-i,\;\; i=0,1,...,\nu \}$ and $d=min\{b_{i}-i,\;\; i=0,1,...,\nu\}$, are called the height $h$ and the depth $d$ of $D$, respectively. Since $D(x^{n})$ is a polynomial of degree at most $n+h$ and lowest nonzero term with power of $x$ at least $n+d$, it follows that (\ref{eq-genp1}) can also be written in the form
\begin{equation} \label{eq-genp2}
          D(x^{n})=\sum_{k=d}^{h}\alpha_{n\;k+n}x^{k+n}.
\end{equation}
It is clear that if $h<0$, then $\alpha_{n\;k+n}=0$ for all $n<-h$. Let us call $\Pi=(\alpha_{nm})_{(n,m) \in \naturals\times \naturals}$ the matrix representation of $D$, relative to $(\mbox{\boldmath $\mathrm{x}$},\mbox{\boldmath $\mathrm{x}$})$, that is a row and column finite matrix. As the coefficients of $D(x^{n})$ occupy the entries of the rows of $\Pi$, it follows from (\ref{eq-genp2}) that the coefficients $\alpha_{n\;k+n}$, for $k<d$ or $k>h$ are all zero, whence $\Pi$ is a banded infinite matrix. 

Let us define $\xi: \naturals \mapsto {\reals}$ such that
$\xi(n)=\alpha_{n\;h+n}$. The terms of the sequence $(\xi(n))_{n\in \naturals}$\vspace{0.04in}, are the entries of the super-diagonal of $\Pi$. Moreover, it follows from (\ref{eq-genp1}) that the mapping $\xi$ is a real valued polynomial function in $n\in \naturals$. Accordingly $\Omega =\{n\in \naturals:\;\xi(n)=0\}$ is a finite set or empty. If $\Omega \neq \emptyset$, we then define $N=max\;\Omega$. If  $\Omega=\emptyset$, we then assign $N=-1$. In both cases we have $\xi(n)\not=0$ for all $n>N$, whence $\deg\: D(x^{n})=n+h$ for all $n>N$. 
The latter relation shows that the set, $S$, of inaccessible index associated with $D$, relative to $\mbox{\boldmath $\mathrm{x}$}$, is a subset of the finite initial interval of $\naturals$ determined by $N+h+1$, that is $\I(N+h+1)=\{0,1,...,N+h\}$. Applying (\ref{eq-can2}) with $\mbox{\boldmath $\mathrm{x}$}$ in place of $\mbox{\boldmath $\mathrm{b}$}$ and using polynomial notation, we recover Ortiz' original definition for canonical and residual polynomials
\begin{equation}\label{eq-can3} 
 D(\mbox{\boldmath $\mathrm{q}$}_{m}(x))=x^{m}+\mbox{\boldmath $\mathrm{r}$}_{m}(x), \ \ \  (m\in S^{\prime},\ \ \  \mbox{\boldmath $\mathrm{r}$}_{m}(x)\in \mathcal{R_{S}}).
\end{equation}

At this point let us recall a definition from linear algebra \cite[pp. 204]{Bou:Alg}.

\begin{definition} \label{def-decomp} Let $f\in \mbox{$\mathrm{Hom}$}_{\mathcal{F}}(X,Y)$ and $X_{1}, X_{2}$ be subspaces of $X$, let also $Y_{1}, Y_{2}$ be subspaces of $Y$ such that $X=X_{1}\oplus X_{2}$, $Y=Y_{1}\oplus Y_{2}$, where ``$\oplus$" stands for the internal direct sum. If $f_{k}$ is the restriction of $f$ to $X_{k}$ such that $f_{k}(X_{k})\subset Y_{k}$ for $k=1,2$, then we say that $f$ is the direct sum of $f_{k}$ and we write for it $f=f_{1}\oplus f_{2}$.
\end{definition}
Immediate consequences of the above definition are:\\
\[\mbox{$\mathrm{Ker}$}(f)=\mbox{$\mathrm{Ker}$}(f_{1})\oplus \mbox{$\mathrm{Ker}$}(f_{2})\  \mbox{$\mathrm{and}$}\  \mbox{$\mathrm{Im}$}(f)=\mbox{$\mathrm{Im}$}(f_{1})\oplus \mbox{$\mathrm{Im}$}(f_{2}). \]

Let $\mbox{\boldmath $\mathrm{z}$}=(\mbox{\boldmath $\mathrm{z}$}_{n}(x))_{n\in \naturals}$ be a sequence of polynomials. The subspace of ${\mathcal{P}}({\reals})$ spanned by $\mbox{\boldmath $\mathrm{z}$}$ will be denoted by $span(\mbox{\boldmath $\mathrm{z}$})$.
\begin{theorem}\label{theo-decomp} Let $D\in \mathcal{D}$. Let also $\mbox{\boldmath $\mathrm{x}$}^{(1)}=\{x^{n},\; 0\leq n\leq N\}$, $\mbox{\boldmath $\mathrm{x}$}^{(2)}=\{x^{n}, \; n> N\}$, $\mbox{\boldmath $\mathrm{d}$}^{(1)}=\{x^{n}, \; 0\leq n\leq N+h\}$, $\mbox{\boldmath $\mathrm{d}$}^{(2)}=\{D(x^{n}), \; n> N\}$. Then we have\emph{:}\\
\emph{i)} ${\mathcal{P}}({\reals})=span(\mbox{\boldmath $\mathrm{x}$}^{(1)})\oplus span(\mbox{\boldmath $\mathrm{x}$}^{(2)})$ and ${\mathcal{P}}({\reals})=span(\mbox{\boldmath $\mathrm{d}$}^{(1)})\oplus span(\mbox{\boldmath $\mathrm{d}$}^{(2)})$.\\
\emph{ii)} Let $D_{k}$ be the restriction of $D$ to $span(\mbox{\boldmath $\mathrm{x}$}^{(k)})$ $(k=1,2)$. Then $D=D_{1}\oplus D_{2}$.\\
\emph{iii)} $\mbox{$\mathrm{Ker}$}(D)=\mbox{$\mathrm{Ker}$}(D_{1})$.\\ 
\emph{iv)} Let $S$ be the set of inaccessible index associated with $D$, relative to $\mbox{\boldmath $\mathrm{x}$}$ and $S_{1}$ the set of inaccessible index associated with $D_{1}$ relative to $\mbox{\boldmath $\mathrm{d}$}^{(1)}$. Then ${\mathcal{R}}_{S}={\mathcal{R}}_{S_{1}}$.
\end{theorem}
\begin{proof} i) As  $\mbox{\boldmath $\mathrm{x}$}^{(1)}\cap \mbox{\boldmath $\mathrm{x}$}^{(2)}=\emptyset$ and $\mbox{\boldmath $\mathrm{x}$}=\mbox{\boldmath $\mathrm{x}$}^{(1)} \cup \mbox{\boldmath $\mathrm{x}$}^{(2)}$ is the basis $\{x^{n},\;\;n\in \naturals\}$, it follows that ${\mathcal{P}}({\reals})=span(\mbox{\boldmath $\mathrm{x}$}^{(1)})\oplus span(\mbox{\boldmath $\mathrm{x}$}^{(2)})$. Clearly $\mbox{\boldmath $\mathrm{d}$}^{(1)} \cap \mbox{\boldmath $\mathrm{d}$}^{(2)}=\emptyset$.
Let us define the sequence of polynomials $\mbox{\boldmath $\mathrm{d}$}=(\mbox{\boldmath $\mathrm{d}$}_{n}(x))_{n\in \naturals}$ as follows
\[\mbox{\boldmath $\mathrm{d}$}_{n}(x)=\left\{ \begin{array}{ll} x^{n}, & \mbox{if $n\leq N+h$} \\
Dx^{n-h}, & \mbox{if $n>N+h$}, 
	\end{array} \right. 
\]
Formally, $\mbox{\boldmath $\mathrm{d}$}^{(2)}=\{Dx^{n-h}, \; n>N +h\}$, whence $\mbox{\boldmath $\mathrm{d}$}=\mbox{\boldmath $\mathrm{d}$}^{(1)} \cup \mbox{\boldmath $\mathrm{d}$}^{(2)}$. Moreover the terms of $\mbox{\boldmath $\mathrm{d}$}$ satisfy the relation $gsup(\mbox{\boldmath $\mathrm{d}$}_{n}(x))=\deg\: \mbox{\boldmath $\mathrm{d}$}_{n}(x)=n$ for all $n\in \naturals$. Applying (\ref{eq-prestand}) with $S=\emptyset$, we deduce that $\mbox{\boldmath $\mathrm{d}$}$ is a basis of ${\mathcal{P}}({\reals})$ and so ${\mathcal{P}}({\reals})=span(\mbox{\boldmath $\mathrm{d}$}^{(1)})\oplus span(\mbox{\boldmath $\mathrm{d}$}^{(2)})$.\\
ii) As $D_{k}(span(\mbox{\boldmath $\mathrm{x}$}^{(k)})) \subset span(\mbox{\boldmath $\mathrm{d}$}^{(k)})$ $(k=1,2)$, Definition \ref{def-decomp} is fulfilled, whence the assertion.\\
iii) It follows from (ii) that $\mbox{$\mathrm{Ker}$}(D)=\mbox{$\mathrm{Ker}$}(D_{1})\oplus \mbox{$\mathrm{Ker}$}(D_{2})$. As $\mbox{\boldmath $\mathrm{x}$}^{(2)}\ni x^{n}\mapsto D(x^{n})\in \mbox{\boldmath $\mathrm{d}$}^{(2)}$ is a bijection of bases the mapping $D_{2}: span(\mbox{\boldmath $\mathrm{x}$}^{(2)}) \mapsto span(\mbox{\boldmath $\mathrm{d}$}^{(2)})$ is an isomorphism and so $\mbox{$\mathrm{Ker}$}(D_{2})=\{0\}$, whence the assertion.\\ 
iv) Formally $S_{1}=\{s \in \naturals: s\in \I(N+h+1)\wedge s\not=gsup(\mbox{\boldmath $\mathrm{g}$})\; \forall \mbox{\boldmath $\mathrm{g}$}\in \mbox{$\mathrm{Im}$}(D_{1})\}$. Let $s$ be an arbitrary element of $S$, then $s\in \I(N+h+1)$ and $s\not=gsup(\mbox{\boldmath $\mathrm{g}$})\ \forall \mbox{\boldmath $\mathrm{g}$}\in \mbox{$\mathrm{Im}$}(D)$. Since $\mbox{$\mathrm{Im}$}(D_{1})\subset \mbox{$\mathrm{Im}$}(D)$ we also have $s\not=gsup(\mbox{\boldmath $\mathrm{g}$})\  \forall \mbox{\boldmath $\mathrm{g}$}\in \mbox{$\mathrm{Im}$}(D_{1})$. As $s$ satisfies both properties of $S_{1}$ it follows that $S\subset S_{1}$ and so ${\mathcal{R}}_{S}\subset {\mathcal{R}}_{S_{1}}$. Taking into account that $D_{2}$ is an isomorphism, it turns out that $\mbox{$\mathrm{Im}$}(D_{2})=span(\mbox{\boldmath $\mathrm{d}$}^{(2)})$. As the internal direct sum of vector subspaces is associative and commutative we have 
\[\begin{array}{lll}{\mathcal{P}}({\reals})&=&span(\mbox{\boldmath $\mathrm{d}$}^{(1)})\oplus span(\mbox{\boldmath $\mathrm{d}$}^{(2)})\\
&=&(\mbox{$\mathrm{Im}$}(D_{1})\oplus {\mathcal{R}}_{S_{1}})\oplus \mbox{$\mathrm{Im}$}(D_{2})\\
&=&(\mbox{$\mathrm{Im}$}(D_{1})\oplus \mbox{$\mathrm{Im}$}(D_{2}))\oplus {\mathcal{R}}_{S_{1}}\\
&=&\mbox{$\mathrm{Im}$}(D)\oplus {\mathcal{R}}_{S_{1}},
\end{array}\]
which shows that ${\mathcal{R}}_{S_{1}}$ is an algebraic complement of $\mbox{$\mathrm{Im}$}(D)$. Moreover, as 
${\mathcal{R}}_{S}$ is also an algebraic complement of $\mbox{$\mathrm{Im}$}(D)$ with ${\mathcal{R}}_{S}\subset {\mathcal{R}}_{S_{1}}$, we infer that ${\mathcal{R}}_{S}={\mathcal{R}}_{S_{1}}$ as required.
\end{proof}
Let $f :X\mapsto Y$ be a linear operator. Following the classical terminology and notation we define the rank of $f$: $rank(f)=dim\; \mbox{$\mathrm{Im}$}(f)$; the \emph{Kernel index} of $f$: $\alpha(f)=dim\; \mbox{$\mathrm{Ker}$}(f)$; the \emph{deficiency} of $f$ (or the codimension of Im($f$)): $\beta(f)=dim {\mathcal{R}}_{S}$, that is $\beta(f)=card(S)$; and the index of $f$: $index(f)=\alpha(f)-\beta(f)$. Let us recall from linear algebra that if $X, Y$ are finite dimensional spaces, then $rank(f)=dim(Y)-\beta(f)=dim(X)-\alpha(f)$.
Using the notation of Theorem \ref{theo-decomp}, let us notice that $D$ is of infinite rank but of finite index and $D_{1}$ is both of finite index and of finite rank.
\begin{theorem}\label{theo-height}The height $h$ of a differential operator $D\in \mathcal{D}$ satisfies
\[h=-index(D)\]
\end{theorem}
\begin{proof}  As $D_{1}$ maps $span(\mbox{\boldmath $\mathrm{x}$}^{(1)})$ into $span(\mbox{\boldmath $\mathrm{d}$}^{(1)})$, it follows that
\begin{equation} \label{eq-rank1}
rank(D_{1})=card(\mbox{\boldmath $\mathrm{x}$}^{(1)})-\alpha(D_{1})=N+1-\alpha(D_{1}). 
\end{equation}
Moreover, since $card(\mbox{\boldmath $\mathrm{d}$}^{(1)})=rank(D_{1})+\beta(D_{1})$, it follows form (\ref{eq-rank1}) that
\begin{equation} \label{eq-index1}
	  card(\mbox{\boldmath $\mathrm{d}$}^{(1)})=N-\alpha(D_{1})+\beta(D_{1})+1.   
\end{equation}
Also the definition of  $\mbox{\boldmath $\mathrm{d}$}^{(1)}$ implies
\begin{equation} \label{eq-rank2}
card(\mbox{\boldmath $\mathrm{d}$}^{(1)})=N+h+1. 
\end{equation}
We thus infer from (\ref{eq-index1}) and (\ref{eq-rank2}) that $h=\beta(D_{1})-\alpha(D_{1})$. On account of Theorem \ref{theo-decomp} (iii) and (iv), i.e. $\alpha(D_{1})=\alpha(D)$ and $\beta(D_{1})=\beta(D)$, it follows immediately that
$h=\beta(D)-\alpha(D)=-index(D)$, as claimed.
\end{proof}
\section{Row Echelon Transformations}
\label{sec:TheEchelonTransformationOfTheMatrixRepresentationOfODEs}
Let us recall some useful results from linear algebra.
\paragraph{Remarks.} Let  $\xi=(\xi_{1},...,\xi_{n})$ be a list of $n$ vectors of a vector space $X$.
Certainly the process of replacing a vector of a list by a nonzero linear combination of any vectors of the same list can be described by a composite of a finite sequence of elementary operations. Whereas a composite $\Phi_{X}$ of elementary operations, applied to lists of $X$ is an automorphism on $X^{n}$, it preserves the linear independence and the span of the original list of vectors.  Let $\Phi_{Y}$ be the composite of the same elementary operators but applied to a list of $n$ vectors from another vector space $Y$. It turns out that any composite of elementary operations commutes with linear mappings in the sense that if $f : X \mapsto Y$ is a linear mapping, then
\begin{equation}\label{eq-remark}
                f((\Phi_{X}(\xi_{1},...,\xi_{n}))_{j})=(\Phi_{Y}(f(\xi_{1}),...,f(\xi_{n})))_{j}
\end{equation} 
for all $j=1,2,...,n$, where $(...)_{j}$ stand for the $j$-th vector of the corresponding list. 
An $n \times m$ matrix $\alpha=(\alpha_{ij})$ can be regarded as a list of row vectors, say $(\vec{\alpha}_{1},...,\vec{\alpha}_{n})^T$ of the vector space ${\mathcal{F}}^{m}$. Considering the canonical ordered basis of ${\mathcal{F}}^{m}$, let us call $\sigma_{i}=gsup(\vec{\alpha}_{i})$, relative to the the standard ordering ``$<$" by magnitude of $\naturals$. If $supp(\vec{\alpha}_{i})=\emptyset$, then $\vec{\alpha}_{i}$ is the zero vector and we naturally assign $\sigma_{i}$ to zero. According to Definition \ref{def-ech}, if a matrix $\alpha=(\alpha_{ij})$ is in LREF, then $\sigma_{i} < \sigma_{i+1}$ for all $i$ such that $\sigma_{i}>0$. Also, zero rows, if any, appear on the top of a matrix in LREF. Every finite matrix can be transformed into an LREF by means of a composite of a finite sequence of elementary operations. Let $X,Y$ be vector spaces of dimensions $n,m$ respectively. Let also $\alpha=(\alpha_{ij})$ be the matrix representation of a linear mapping $f:X \mapsto Y$ relative to a pair of ordered bases, say $(\mbox{\boldmath{$\mathrm{\chi}$}},\mbox{\boldmath{$\mathrm{b}$}})$. Transforming $\alpha$ into an LREF, say $\pi=\Phi(\alpha)$, it follows (in accordance with Proposition \ref{lem-prop}) that the nonzero rows of $\pi$ determine a basis of the range of $f$ and the set of inaccessible index, $S$, associated with $f$, relative to $\mbox{\boldmath{$\mathrm{b}$}}$, is given by $S=\{1,...,m\}\backslash \{\sigma_{i}, i=1,...,n\}$, where $\sigma_{i}=gsup(\vec{\pi}_{i})$.
Moreover, applying the same $\Phi$ to the basis $\mbox{\boldmath{$\mathrm{\chi}$}}$, by virtue of (\ref{eq-remark}), a new basis $\mbox{\boldmath $\mathrm{s}$}=\Phi(\mbox{\boldmath{$\mathrm{\chi}$}})$ of $X$ is generated, that is a standard basis associated with $f$ relative to $\mbox{\boldmath{$\mathrm{b}$}}$, which includes a basis of the null space of $f$. Formally the matrix representation of $f$ relative to $(\mbox{\boldmath $\mathrm{s}$},\mbox{\boldmath{$\mathrm{b}$}})$ is $\pi$.

Following the notation of Theorem \ref{theo-decomp}, $\Pi_{1}$ will stand for the matrix representation of $D_{1}$, relative to ($\mbox{\boldmath $\mathrm{x}$}^{(1)}$, $\mbox{\boldmath $\mathrm{x}$}$), which has a finite number of rows and an infinite number of zero columns. Moreover $\Pi_{2}$ will stand for the matrix representation of $D_{2}$ relative to ($\mbox{\boldmath $\mathrm{x}$}^{(2)}$, $\mbox{\boldmath $\mathrm{x}$}$). We can write the matrix representation of $D$ relative to ($\mbox{\boldmath $\mathrm{x}$}$, $\mbox{\boldmath $\mathrm{x}$}$) as a two-blocks matrix 
\begin{equation}\label{eq-bandD}
\D_{\mbox{\boldmath $\mathrm{x}$}}^{\mbox{\boldmath $\mathrm{x}$}}=\left(\begin{array}{c} \Pi_{1} \\
\hline\\
\Pi_{2} \\
\end{array} \right)
\end{equation}  
An echelon transformation on $\Pi_{1}$ requires a finite sequence of row elementary operations for reducing $\Pi_{1}$ into a pre-LREF, say $\Pi^{*}_{1}$. Applying the same sequence of elementary operations to the domain basis $\mbox{\boldmath $\mathrm{x}$}^{(1)}=\{x^{n},\ 0\leq n\leq N\}$, we derive a standard basis $\mbox{\boldmath $\mathrm{s}$}^{(1)}=\{\mbox{\boldmath $\mathrm{s}$}_{n}(x),\: 0\leq n\leq N\}$ associated with $D_{1}$ relative to $\mbox{\boldmath $\mathrm{x}$}$. Let us call  $\mbox{\boldmath $\mathrm{s}$}=\mbox{\boldmath $\mathrm{s}$}^{(1)}\cup \mbox{\boldmath $\mathrm{x}$}^{(2)}$. Since $\mbox{\boldmath $\mathrm{s}$}$ is the union of bases of factor subspaces of ${\mathcal{P}}({\reals})$, it is a basis of ${\mathcal{P}}({\reals})$ too. Certainly the matrix representation of $D$, relative to ($\mbox{\boldmath $\mathrm{s}$}, {\mbox{\boldmath $\mathrm{x}$}}$) is given by
\begin{equation}\label{eq-lrefD}
\D_{\mbox{\boldmath $\mathrm{s}$}}^{\mbox{\boldmath $\mathrm{x}$}}=\left(\begin{array}{c} \Pi_{1}^{*}\\
\hline\\
\Pi_{2} \\
\end{array} \right)
\end{equation}  
\begin{theorem}\label{theo-ref} The matrix representation of $D\in \mathcal{D}$ relative to \emph{(}$\mbox{\boldmath $\mathrm{s}$}$, $\mbox{\boldmath $\mathrm{x}$}$\emph{)}, given by \emph{(\ref{eq-lrefD})}, is in pre-\emph{LREF}.
\end{theorem}
\begin{proof} In view of Theorem \ref{theo-decomp} $\mbox{$\mathrm{Ker}$}(D)=\mbox{$\mathrm{Ker}$}(D_{1})$, the zero rows of $\D_{\mbox{\boldmath $\mathrm{s}$}}^{\mbox{\boldmath $\mathrm{x}$}}$\vspace{0.01in} are the zero rows of $\Pi_{1}^{*}$. Moreover, as $\Pi_{1}^{*}$ is in LREF, it turns out that the property (ii) in Definition \ref{def-ech} is fulfilled. It remains to show that property (iii) in Definition \ref{def-ech} holds true, namely that the mapping $J\ni j\mapsto \sigma_{j}=\deg\:D(\mbox{\boldmath $\mathrm{s}$}_{j}(x))\in \naturals$ is an order monomorphism. Let $N\in\naturals$. Let also $J_{1}$ be the indexing set of the nonzero rows of $\Pi_{1}^{*}$, i.e. $J_{1}=\{j\in J : 0\leq j \leq N\}$. As $\Pi_{1}^{*}$ is in LREF, it follows that the restriction of $\sigma$ to $J_{1}$ is an order monomorphism. Thus if $k,m\in J_{1}$ and $k<m$, then $\sigma_{k}<\sigma_{m}$. Let call $J_{1}^{\prime}=J\verb+\+J_{1}$ that is $J_{1}^{\prime}=\{j\in J : j>N\}$. The restriction of $\sigma$ to $J_{1}^{\prime}$ is given by $\sigma_{j}=j+h$ and so it is formally an order monomorphism. Thus if $k,m\in J_{1}^{\prime}$ and $k<m$, then $\sigma_{k}<\sigma_{m}$. Let us finally consider the case when $k\in J_{1}$ and $m\in J_{1}^{\prime}$. As $N$ is the last element of $J_{1}$ and $N+1$ is the first element of $J_{1}^{\prime}$ it follows that $k\leq N < N+1 \leq m$ and $\sigma_{k}\leq \sigma_{N}< N+h < N+h+1=\sigma_{N+1} \leq \sigma_{m}$, whence the assertion. If $N \not\in \naturals$, then $\sigma_{j}=j+h$ for all $j\in J$ and the proof is complete.\end{proof}

As a consequence of Theorem \ref{theo-ref} the basis $\mbox{\boldmath $\mathrm{s}$}=\mbox{\boldmath $\mathrm{s}$}^{(1)}\cup \mbox{\boldmath $\mathrm{x}$}^{(2)}$ of ${\mathcal{P}}({\reals})$ is a standard basis associated with $D$ relative to $\mbox{\boldmath $\mathrm{x}$}$. It turns out that a basis of the polynomial kernel of $D$ is an initial part of $(\mbox{\boldmath $\mathrm{s}$}_{n}(x))_{0\leq n\leq N}$, say $\mbox{\boldmath $\mathrm{u}$}=(\mbox{\boldmath $\mathrm{s}$}_{n}(x))_{n\in W}$, such that $D_{1}(\mbox{\boldmath $\mathrm{s}$}_{n}(x))=0$ and the set of inaccessible index associated with $D$, relative to $\mbox{\boldmath $\mathrm{x}$}$, is given by 
\begin{equation}\label{eq-S1}
S=\I(N+h+1)\verb+\+\{\deg\: D_{1}(\mbox{\boldmath $\mathrm{s}$}_{n}(x)):\;0\leq n\leq N\}. 
\end{equation}
Formally $\mbox{\boldmath $\mathrm{s}$}\verb+\+\mbox{\boldmath $\mathrm{u}$}$ determines a standard sequence associated with $D$ relative to $\mbox{\boldmath $\mathrm{x}$}$. 

We thus conclude that in all cases of ODEs determined by differential operators in the class $\mathcal{D}$, a standard basis $\mbox{\boldmath $\mathrm{s}$}$ associated with $D\in \mathcal{D}$, relative to $\mbox{\boldmath $\mathrm{x}$}$, is automatically constructed and so a complete sequence of CPs, say $(\mbox{\boldmath $\mathrm{q}$}_{m}(x))_{m\in S^{\prime}}$, is directly generated by (\ref{eq-cre1}). Moreover a basis of CPs is given by $\tilde{\mbox{\boldmath $\mathrm{q}$}}=(\mbox{\boldmath $\mathrm{s}$}_{n}(x))_{n\in W}\cup (\mbox{\boldmath $\mathrm{q}$}_{m}(x))_{m\in S^{\prime}}$.  The elements of  $\mbox{\boldmath $\mathrm{u}$}=(\mbox{\boldmath $\mathrm{s}$}_{n}(x))_{n\in W}$ will be referred to as null CPs. The classification of the remaining elements of a basis of CPs, those indexed by $S^{\prime}$, will be discussed in the following section.
\section{Complete Sequences of Canonical Polynomials and Bunchaft's Classification}
\label{sec:CompleteSequencesOfCanonicalPolynomialsAndBunchaft'sClassification}
In this paragraph we discuss Bunchaft's classification of CPs in connection with the concept of ``complete sequence of CPs", introduced earlier in this paper. For these purposes, in this section, $J$ will stand for the indexing set of the nonzero rows of the matrix representation of $D\in \mathcal{D}$, relative to ($\mbox{\boldmath $\mathrm{x}$}$, $\mbox{\boldmath $\mathrm{x}$}$) and $\sigma$ for the corresponding mapping $\sigma: J \ni n\mapsto \sigma_{n}=\deg\: D(x^{n})\in S^{\prime}$.
\begin{theorem}\label{theo-icp} In a complete sequence of \emph{CPs} any term, say $\mbox{\boldmath $\mathrm{q}$}_{m}(x)$, satisfies 
$m=\deg\:D(\mbox{\boldmath $\mathrm{q}$}_{m}(x))$ 
for all $m\in S^{\prime}$.
\end{theorem}
\begin{proof} In view of (\ref{eq-can3}), we must show that $\deg\: \mbox{\boldmath $\mathrm{r}$}_{m}(x)<m$ for all $m\in S^{\prime}$. Let us rewrite the residual polynomial as $\mbox{\boldmath $\mathrm{r}$}_{m}(x)=r_{s_{1}}x^{s_{1}}+r_{s_{2}}x^{s_{2}}+...+r_{s_{k}}x^{s_{k}}$ with $s_{i}\in S$ and $s_{1}<s_{2}<...<s_{k}$. It suffices to show that $m>s_{i}$ for all $i$. On the contrary we assume that $m\leq s_{i}$ for some $i$. Now  $\deg\: D(\mbox{\boldmath $\mathrm{q}$}_{m}(x))=s_{k}\in S$. It turns out that there exists a polynomial in the range of $D$ with greatest element of its support in $S$. The latter contradicts the definition of $S$, whence the assertion follows.
\end{proof}

\begin{definition} \label{def-class} A canonical polynomial $\mbox{\boldmath $\mathrm{q}$}_{m}(x)$ is called primary if there is some $n\in \naturals$ such that $m=\sigma_{n}$. Otherwise it is called derived-singular. A primary \emph{CP} is called primary-generic if further $\sigma_{n}=n+h$. Otherwise it is called primary-singular. 
\end{definition}

The three categories of CPs, indexed by $S^{\prime}$, are displayed in the following diagram:
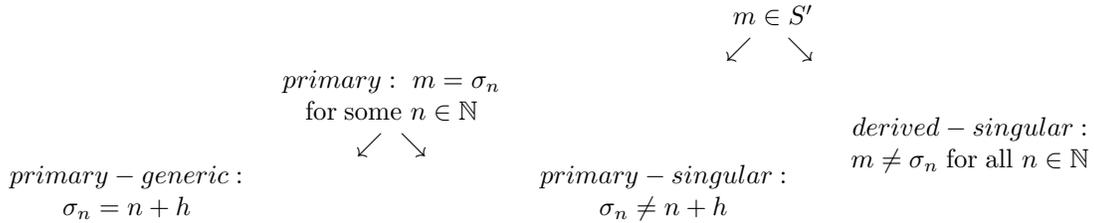
\[\begin{array}{rrl}
m\in S^{\prime} & &  \\
\swarrow \ \  \ \ \searrow & & \\
\begin{array}{ccc} & primary:\ m=\sigma_{n} & \\
& \mbox{$\mathrm{for\; some}$}\ n\in \naturals&  \\
&\swarrow \ \  \searrow&   \\
\begin{array}{c} primary-generic: \\
\sigma_{n}=n+h \end{array} & & \begin{array}{c} primary-singular: \\
\sigma_{n}\not=n+h \end{array}
\end{array}  & & \hspace{-0.15in} \begin{array}{c} derived-singular:\\ m\not=\sigma_{n}\; \mbox{$\mathrm{for\; all}$}\ n\in \naturals\end{array}
\end{array}
\]
It is clear from the above definition and Theorem \ref{theo-icp} that if a derived-singular CP of index $m\in S^{\prime}$, say $\mbox{\boldmath $\mathrm{q}$}_{m}(x)$, exists, then $m\not=\deg\: D(x^{n})$ for all $n\in \naturals$. It amounts to saying precisely that none of the images of the terms of  $\mbox{\boldmath $\mathrm{x}$}$ under $D$ yields the degree $m$; although as $m\in S^{\prime}$ there exist polynomials in $\mbox{$\mathrm{Im}$}(D)$ yielding the degree $m$. In contrast, using in place of $\mbox{\boldmath $\mathrm{x}$}$ any standard basis, its images under $D$ generate all of the degrees $m\in S^{\prime}$ and so the modified recursive formula (\ref{eq-cre1}) results in a complete sequence of CPs. Moreover as $\sigma_{n}=n+h$ for all $n>N$, it follows directly from Definition \ref{def-class} that for any index $n>N$ the  $\mbox{\boldmath $\mathrm{q}$}_{\sigma_{n}}(x)$ is a primary-generic CP. Thus primary-singular and derived-singular CPs form a finite subsequence of a complete sequence of CPs whose indices are all bounded by $N+h$. We directly infer the following criterion.

\begin{theorem}\label{theo-prim}  There exists a complete sequence of \emph{CPs} consisting of primary \emph{(}generic or singular\emph{) CPs} if and only if $\sigma: J\mapsto S^{\prime}$ is surjective. Equivalently, every complete sequence of \emph{CPs} contains at least one derived-singular \emph{CP}, if and only if $\ \mbox{$\mathrm{Im}$}(\sigma)$ is a proper subset of $S^{\prime}$.
\end{theorem}
The following statements enables us to rewrite the above criterion in terms of standard sequences.
\begin{proposition}\label{lem-prim} The usual polynomial basis $\mbox{\boldmath $\mathrm{x}$}$ contains a standard sequence of polynomials associated with $D$, relative to $\mbox{\boldmath $\mathrm{x}$}$, if and only if $\sigma: J\mapsto S^{\prime}$ is surjective. 
\end{proposition}
\begin{proof} Let $\mbox{\boldmath $\mathrm{x}$}^{*}$ be a standard sequence such that $\mbox{\boldmath $\mathrm{x}$}^{*} \subset \mbox{\boldmath $\mathrm{x}$}$. Then $\mbox{\boldmath $\mathrm{x}$}^{*}$ can be extended to a standard basis, say $\mbox{\boldmath $\mathrm{s}$}$, associated with $D$, relative to $\mbox{\boldmath $\mathrm{x}$}$, by means of elements of a basis of $\mbox{$\mathrm{Ker}$}(D)$. As the matrix representation of $D$, relative to $(\mbox{\boldmath $\mathrm{s}$}, \mbox{\boldmath $\mathrm{x}$})$ is in \emph{pre}-LREF the result follows from Proposition \ref{lem-prop}.
Conversely, as $\sigma: J\mapsto S^{\prime}$ is surjective, it follows that for every $m\in S^{\prime}$ there exists some $n\in \naturals$ such that $m=\sigma_{n}$ (or $m=\deg\: D(x^{n})$). For each $m\in S^{\prime}$ we define the set $Z_{m}=\{n\in J: \sigma_{n}=m\}$ and the set ${\mathcal{X}}_{m}=\{x^{n}: n\in Z_{m}\}$. Let us choose an element, say $\mbox{\boldmath $\mathrm{s}$}_{m}(x)$, from each set ${\mathcal{X}}_{m}$. As the terms of the sequence $(\mbox{\boldmath $\mathrm{s}$}_{m}(x))_{m\in S^{\prime}}$ satisfy $\deg\:D(\mbox{\boldmath $\mathrm{s}$}_{m}(x))=m$ for all $m\in S^{\prime}$, it follows from (\ref{eq-standfa1}) that $(\mbox{\boldmath $\mathrm{s}$}_{m}(x))_{m\in S^{\prime}}$ is a standard sequence associated with $D$, relative to $\mbox{\boldmath $\mathrm{x}$}$, whose terms are elements of $\mbox{\boldmath $\mathrm{x}$}$, as claimed. 
\end{proof}

Combining Theorem \ref{theo-prim} with Proposition \ref{lem-prim} we deduce the following result.
\begin{corollary}\label{cor-prim} There exists a complete sequence of \emph{CPs} consisting of primary \emph{(}generic or singular\emph{) CPs} if and only if $\mbox{\boldmath $\mathrm{x}$}$ contains a standard sequence of polynomials associated with $D$, relative to $\mbox{\boldmath $\mathrm{x}$}$.
\end{corollary}

\section{The Full Recursive Implementation of the Tau Method}
\label{sec:TheFullRecursiveConstructionOfTheTauMethod}
Let $D\in \mathcal{D}$. Let also
\begin{equation}\label{eq-diff.eq}
	    Dy(x)=f(x)
\end{equation}
be a differential equation subjected to $\nu$ supplementary conditions (initial, boundary, mixed) and $f(x)$ be a polynomial. The new strategy introduced by the tau-method was essentially to compute the exact polynomial solution to a modified equation derived by adding a perturbation term to the right hand side of the original equation.
In particular, the perturbed polynomial equation is of the form
\begin{equation}\label{eq-pert}
          Dy_{n}(x)=f(x)+H_{n}(x),
\end{equation}
where the perturbation term $H_{n}(x)=\sum_{i=1}^{M}\tau^{(n)}_{i}\rho_{n-i+1}(x)$ is a linear combination of Chebyshev or Legendre polynomials of degree $n-i+1$ with $n\in \naturals$ and $1 \leq i \leq M$, weighted by the $\tau$-parameters. By (\ref{eq-diff.eq}) and  (\ref{eq-pert}) the linearity of $D$ implies that $D(y(x)-y_{n}(x))=H_{n}(x)$. Accordingly the spectral character of the tau method arises from the fact that the perturbation term $H_{n}(x)$ ``minimizes" the quantity $D(y(x)-y_{n}(x))$.
The existence and uniqueness of the polynomial solution of (\ref{eq-pert}) impose the following two types of conditions:
\begin{itemize}
	\item The approximate solution $y_{n}(x)$ must be a polynomial in the range of $D$, that is $f(x)+H_{n}(x)\in \mbox{$\mathrm{Im}$}(D)$.
	According to (\ref{eq-parm}) this type of conditions provides $card(S)$ (or $\beta(D)$) number of equations, which are usually refered to as S-Tau Method Matching conditions (STMC).\vspace{-0.1in}
	\item The approximate solution $y_{n}(x)$ must satisfy exactly the $\nu$ supplementary conditions, that provides $\nu$ number of equations. 
\end{itemize}
Next, using the above results the following 
result is recovered.
\begin{proposition}\label{theo-tauprm}
For any order $n$ of approximation such that $n>N$ the number $M$ of tau parameters to satisfying the required conditions must be:  $M=\nu+h$. 
\end{proposition}
\begin{proof} Taking into account that $\alpha(D)=card(W)$, the total number of parameters ($C_{w}$ and $\tau^{(n)}_{i}$), that is $\alpha(D)+M$, must be equal to the number of equations induced by the required conditions, that is $\nu+\beta(D)$, whence $M=\nu-\alpha(D)+\beta(D)$ for $n>N$. By virtue of Theorem \ref{theo-height}, we infer that $M=\nu+h$ for $n>N$.\end{proof}

Taking $n>N$ in (\ref{eq-pert}) an inspection of the proof of proposition (\ref{theo-tauprm}) shows that the tau approximate solution is determined by solving an $(\nu+\beta(D))\times(\nu+\beta(D))$ algebraic system of linear equations. 

On account of Theorem \ref{theo-prim}, a complete sequence of CPs may contain some derived-singular CPs, since $\sigma: J \ni n\mapsto \sigma_{n}=\deg\: D(x^{n})\in S^{\prime}$ is not necessarily surjective. By Corollary \ref{cor-prim}, the latter amounts to saying precisely that $\mbox{\boldmath $\mathrm{x}$}$ does not contain a standard sequence associated with $D$ relative to $\mbox{\boldmath $\mathrm{x}$}$. In this case (see Example 2) the recursive formulae either in the original or in the extended Bunchaft's form are not efficient for the generation of a basis of CPs. Following the approach discussed in this paper, we avoid this difficulty by starting the  program with an echelon transformation of the matrix representation of $D_{1}$. Such a transformation, on account of Theorem \ref{theo-ref}, also provides an LREF of the matrix representation of the entire $D$, coupled with a standard basis associated with $D$, relative to $\mbox{\boldmath $\mathrm{x}$}$. The modified recursive formula (\ref{eq-cre1}) is then activated so as to generate a basis of CPs. 

Even in the rather frequent case, in which a complete sequence of CPs consists of primary CPs exclusively, an echelon  transformation of the matrix representation of $D_{1}$ is essential for the automatic construction of the set of inaccessible index along with the null CPs and so for the recursive construction of a basis of CPs. It turns out that the above-mentioned change of the programming strategy is necessary for the full recursive implementation of the tau method via bases of CPs. 

Programming Languages which support symbolic computations, such as Mathematica,  make it possible to calculate within a unified program the algebraic entities, introduced earlier in this paper, and eventually to curry out the self-operative recursive implementation of the tau method for differential equations determined by operators in $\mathcal{D}$.
\section{Examples}
The following examples indicate the basic steps, which directly lead to a computer's program executable by Mathematica.
We start with the last example treated by Bunchaft in~\cite{Bu:ext}.
\paragraph{Example 1}
\label{sec:Example1} 
Let the differential operator be
\[
	D=\frac{d^{4}}{dx^{4}}+x\frac{d^{3}}{dx^{3}}-3\frac{d^{2}}{dx^{2}}.
\]
Thus $h=max\{-4, 1-3, 0-2\}=-2$ and $index(D)=2$. A generating system of $\mbox{$\mathrm{Im}$}(D)$ is given by the polynomials\vspace{-0.1in}
\begin{equation}\label{eq-gsyst1}
			D(x^{n}) = n(n - 1)(n - 2)(n - 3)x^{n-4}+n(n-1)(n-5)x^{n-2},
\end{equation}
for $n\in \naturals$. Each row of the matrix $\D_{\mbox{\boldmath $\mathrm{x}$}}^{\mbox{\boldmath $\mathrm{x}$}}$, which represents $D$ relative to $(\mbox{\boldmath $\mathrm{x}$},\;\mbox{\boldmath $\mathrm{x}$})$, is: 
\begin{equation}\label{eq-bmr} 
\left(...,0,n(n - 1)(n - 2)(n - 3),\  0,\  n(n-1)(n-5),0,...\right), n\in \naturals.
\end{equation}
As the leading coefficient of (\ref{eq-bmr}) is $\xi(n)=n(n - 1)(n - 5)$, it follows that the greatest root of $\xi(n)=0$ in $\naturals$ is $N=5$. Now $N$ subdivides the domain basis ${\mbox{\boldmath $\mathrm{x}$}}$ into two disjoint sets ${\mbox{\boldmath $\mathrm{x}$}}=\mbox{\boldmath $\mathrm{x}$}^{(1)}\cup \mbox{\boldmath $\mathrm{x}$}^{(2)}$ and so the expansion of (\ref{eq-bmr}) gives the two-blocks matrix representation of $D$ in the form (\ref{eq-bandD}), that is
\begin{equation} \label{eq-ex1.bmr}  \D_{\mbox{\boldmath $\mathrm{x}$}}^{\mbox{\boldmath $\mathrm{x}$}}=\left(\begin{array}{rrrrrr}
0&0&0&0&0&0.. \\
0&0&0&0&0&0.. \\
-6&0&0&0&0&0.. \\
0&-12&0&0&0&0.. \\   
24&0&-12&0&0&0.. \\
0&120&0&0&0&0.. \\
\hline\\
0&0&360&0&30&0..\\
0&0&0&840&0&84..\\
...
\end{array} \right).\\
\end{equation}
In view of Theorem \ref{theo-ref} a \emph{pre}-LREF of (\ref{eq-ex1.bmr}) in the form (\ref{eq-lrefD}) is derived by applying a sequence of row elementary operations to the top submatrix of (\ref{eq-ex1.bmr}) consisting of the first $N+1=6$ rows of $\D_{\mbox{\boldmath $\mathrm{x}$}}^{\mbox{\boldmath $\mathrm{x}$}}$, that is
\begin{equation} \label{eq-ex1.ref} 
\D_{\mbox{\boldmath $\mathrm{s}$}}^{\mbox{\boldmath $\mathrm{x}$}}=\left(\begin{array}{rrrrrr}
0&0&0&0&0&0.. \\
0&0&0&0&0&0.. \\
0&0&0&0&0&0.. \\
-6&0&0&0&0&0.. \\
0&-12&0&0&0&0.. \\   
24&0&-12&0&0&0.. \\
\hline\\
0&0&360&0&30&0..\\
0&0&0&840&0&84..\\
...
\end{array} \right).
\end{equation}
\\
The standard basis $\mbox{\boldmath $\mathrm{s}$}$ is derived by applying the same sequence of row elementary operations to the corresponding finite part of the domain basis $\mbox{\boldmath $\mathrm{x}$}$, whence ${\mbox{\boldmath $\mathrm{s}$}}=(1, x, x^{5}+10x^{3}, x^{2}, x^{3}, x^{4}, x^{6}, x^{7}, ...)$. The first three terms of $\mbox{\boldmath $\mathrm{s}$}$ are the null CPs, which correspond to the zero rows of (\ref{eq-ex1.ref}), forming a basis of $\mbox{$\mathrm{Ker}$}(D)$, that is $\mbox{\boldmath $\mathrm{u}$}=\{1, x, x^{5}+10x^{3}\}$. By  (\ref{eq-ex1.ref}) $\{\deg\: D(\mbox{\boldmath $\mathrm{s}$}_{n}(x)):\;0\leq n\leq 5\}=\{0,1,2\}$ and so (\ref{eq-S1}) implies that $S=\I(4)\verb+\+\{\deg\: D(\mbox{\boldmath $\mathrm{s}$}_{n}(x)):\;0\leq n\leq 5\}=\{3\}$, which means that there is no polynomial of third degree in $\mbox{$\mathrm{Im}$}(D)$. The set $S$ also follows directly from (\ref{eq-ex1.ref}), taking into account that all but the column position $j=3$ (forth column), correnspond to accessible degrees. We thus conclude that the residual space is ${\mathcal{R}}_{S}=span\{x^{3}\}$. As the resulting standard sequence associated with $D$ relative to $\mbox{\boldmath $\mathrm{x}$}$, that is $(x^{2}, x^{3}, x^{4}, x^{6}, x^{7}, ...)$, is contained in $\mbox{\boldmath $\mathrm{x}$}$ it follows from Corollary \ref{cor-prim} that a complete sequence of CPs, say $\mbox{\boldmath $\mathrm{q}$}$, exclusively consists of primary CPs. Moreover the terms of the derived standard sequence can be described by $\{x^{n}, n\geq 2\  \mbox{$\mathrm{and}$}\  n\not=5\}$ and they satisfy $\deg\: D(x^{n})=n+h$. Accordingly, as $\sigma_{n}=n+h$, all CPs in $\mbox{\boldmath $\mathrm{q}$}$ are primary-generic. By (\ref{eq-cre1}), on account of (\ref{eq-ex1.ref}), the terms of $\mbox{\boldmath $\mathrm{q}$}$ are given by
\[ \begin{array}{l} 
\mbox{\boldmath $\mathrm{q}$}_{0}(x)=-\frac{1}{6}x^{2}\vspace*{0.06in},\\ 
\mbox{\boldmath $\mathrm{q}$}_{1}(x)=-\frac{1}{12}x^{3}\vspace*{0.06in}, \\ 
\mbox{\boldmath $\mathrm{q}$}_{2}(x)=-\frac{1}{12}[x^{4}-24\mbox{\boldmath $\mathrm{q}$}_{0}(x)]\vspace*{0.06in}, \\
\mbox{\boldmath $\mathrm{q}$}_{4}(x)=\frac{1}{30}[x^{6}-360\mbox{\boldmath $\mathrm{q}$}_{2}(x)] \vspace*{0.06in},\\
\mbox{\boldmath $\mathrm{q}$}_{5}(x)=\frac{1}{84}[x^{7}-840\mbox{\boldmath $\mathrm{q}$}_{4}(x)]\vspace*{0.06in},\\
... 
\end{array} 
\]
Let us rewrite the terms of the standard sequence derived above as $\{x^{m+2},\ m\in \naturals\verb+\+\{3\}\}$. 
Now the entries of the nonzero rows of (\ref{eq-ex1.ref}) are occupied by the coefficients of the right hand side polynomials in
\[ \begin{array}{lr}
\!\!\!D(x^{m+2})\!\!=\!(m+2)(m+1)m(m-1)x^{m-2}\!\!+\!(m+2)(m+1)(m-3)x^{m}\!\!\!, \!\!& \mbox{if $m \not=3$}.
\end{array}                        
\]
Thus (\ref{eq-cre1}) takes the form \\
\[\begin{array}{lr} \mbox{\boldmath $\mathrm{q}$}_{m}(x)=
\mbox{$\frac{1}{(m+2)(m+1)(m-3)}$}[x^{m+2}-\mbox{$(m+2)(m+1)m(m-1)$} \mbox{\boldmath $\mathrm{q}$}_{m-2}(x)], &\!\! \mbox{if $m \not=3$}
\end{array}
\]
which coincides with Ortiz' original recursive formula. A basis  $\tilde{\mbox{\boldmath $\mathrm{q}$}}$ of CPs is given by $\tilde{\mbox{\boldmath $\mathrm{q}$}}=\mbox{\boldmath $\mathrm{u}$}\cup \mbox{\boldmath $\mathrm{q}$}$.

If we use in place of the fifth row in (\ref{eq-ex1.ref}) the row $(0,120,0,0,...)$ we derive a row equivalent matrix representation of $D$  that is also in LREF. This row is generated by (\ref{eq-gsyst1}) for $n=5$, being the sixth row of (\ref{eq-ex1.bmr}), and so the corresponding standard basis must be $(1, x, x^{5}+10x^{3}, x^{2}, x^{5}, x^{4}, x^{6}, x^{7}, ...)$. As $\sigma_{5}=\deg\: D(x^{5}) =1$, we deduce that $\sigma_{5}\not=h+5$, which means that the corresponding polynomial $\mbox{\boldmath $\mathrm{q}$}^{*}_{1}(x)=\frac{1}{120}x^{5}$\vspace*{0.04in} is a primary-singular CP of index $m=1$. Following Bunchaft's notation of multiple CPs, we can also write $\mbox{\boldmath $\mathrm{q}$}_{13}(x)$ in place of $\mbox{\boldmath $\mathrm{q}$}_{1}(x)$ and $\mbox{\boldmath $\mathrm{q}$}_{15}(x)$ in place of $\mbox{\boldmath $\mathrm{q}$}^{*}_{1}(x)$, and so Corollary \ref{cor-uniq} implies $\mbox{\boldmath $\mathrm{q}$}_{13}(x)-\mbox{\boldmath $\mathrm{q}$}_{15}(x)\in \mbox{$\mathrm{Ker}$}(D)$, as it can be immediately verified. An alternative complete sequence of CPs, say $\mbox{\boldmath $\mathrm{q}$}^{*}$, is generated by 
\[  \mbox{\boldmath $\mathrm{q}$}^{*}_{m}(x)=\left\{\begin{array}{ll}\!\!  \mbox{$\frac{1}{(m+2)(m+1)(m-3)}$}[x^{m+2}-\mbox{$(m+2)(m+1)m(m-1)$} \mbox{\boldmath $\mathrm{q}$}^{*}_{m-2}(x)], &\!\!\!\!  \mbox{if $m \not=1,3$}\\
\\
\!\!\!\! \frac{1}{120}x^{5}, &\mbox{if $m=1$}\end{array}\right. 
\]
for all $m\in S^{\prime}$. We conclude that $\mbox{\boldmath $\mathrm{q}$}^{*}$ consists of primary-generic CPs (for all $m\not=1,3$) and one primary-singular CP of index $m=1$. In this case a basis of CPs is given by $\mbox{\boldmath $\mathrm{u}$}\cup \mbox{\boldmath $\mathrm{q}$}^{*}$.
\paragraph{Example 2}
\label{sec:Example2}
Let the differential operator be
\begin{equation}\label{eq-DO2}
	D=(x^{2}+1)\frac{d^{4}}{dx^{4}}+(1-3x)\frac{d^{3}}{dx^{3}}+3\frac{d^{2}}{dx^{2}}.
\end{equation}
Then $h=max\{2-4, 1-3, 0-2\}=-2$ and a generating system of $\mbox{$\mathrm{Im}$}(D)$ is formed by the polynomials
\begin{equation}\label{eq-gsyst2}
	\begin{array}{ll}	D(x^{n}) =  n(n - 1)(n - 2)(n - 3)x^{n-4}\!\!\!\! &+ n(n - 1)(n - 2)x^{n-3}\\
		                                                       \!\!\!\! &+n(n-1)(n-3)(n-5)x^{n-2}
		                                                           \end{array}  
\end{equation}
for $n\in \naturals$. The leading coefficient of $D(x^{n})$ is $\xi(n)=n(n - 1)(n - 3)(n - 5)$ and so $N=5$. The  matrix representation of $D$ relative to $(\mbox{\boldmath $\mathrm{x}$},\;\mbox{\boldmath $\mathrm{x}$})$, as a two-blocks matrix, is displayed below 
\\
\begin{equation} \label{eq-ex2.bmr} 
 \D_{\mbox{\boldmath $\mathrm{x}$}}^{\mbox{\boldmath $\mathrm{x}$}}=\left(\begin{array}{rrrrrr}
0&0&0&0&0&0.. \\
0&0&0&0&0&0.. \\
6&0&0&0&0&0.. \\
6&0&0&0&0&0.. \\   
24&24&-12&0&0&0.. \\
0&120&60&0&0&0.. \\
\hline\\
0&0&360&120&90&0..\\
0&0&0&840&210&336..\\
...
\end{array}\right).
\end{equation}
By applying a sequence of row elementary operators to the top submatrix, as designated in (\ref{eq-ex2.bmr}), a matrix representation of $D$ in \emph{pre}-LREF, is of the form
\begin{equation} \label{eq-ex2.ref} 
	 \D_{\mbox{\boldmath $\mathrm{s}$}}^{\mbox{\boldmath $\mathrm{x}$}}=\left(\begin{array}{rrrrrr}
	0&0&0&0&0&0.. \\
	0&0&0&0&0&0.. \\
	0&0&0&0&0&0.. \\
	6&0&0&0&0&0.. \\   
	120&240&0&0&0&0.. \\
	24&24&-12&0&0&0.. \\
	\hline\\
	0&0&360&120&90&0..\\
	0&0&0&840&210&336..\\
	...
	\end{array} \right).
\end{equation}
The corresponding standard basis is ${\mbox{\boldmath $\mathrm{s}$}}=(1, x, x^{2}-x^{3}, x^{2}, 5x^{4}+x^{5}, x^{4}, x^{6}, x^{7}, ...)$. The null CPs form a basis of $\mbox{$\mathrm{Ker}$}(D)$, that is $\mbox{\boldmath $\mathrm{u}$}=\{1,x,x^{2}-x^{3}\}$. It follows directly from (\ref{eq-ex2.ref}) that $S=\{3\}$ and ${\mathcal{R}}_{S}=span\{x^{3}\}$. By (\ref{eq-cre1}), on account of (\ref{eq-ex2.ref}), a complete sequence $\mbox{\boldmath $\mathrm{q}$}$ of CPs is automatically generated by
\[  \begin{array}{l} 
\mbox{\boldmath $\mathrm{q}$}_{0}(x)=\frac{1}{6}x^{2}\vspace*{0.06in},\\ 
\mbox{\boldmath $\mathrm{q}$}_{1}(x)=\frac{1}{240}[5x^{4}+x^{5}-120\mbox{\boldmath $\mathrm{q}$}_{0}(x)]\vspace*{0.06in}, \\ 
\mbox{\boldmath $\mathrm{q}$}_{2}(x)=-\frac{1}{12}[x^{4}-24\mbox{\boldmath $\mathrm{q}$}_{1}(x)-24\mbox{\boldmath $\mathrm{q}$}_{0}(x)]\vspace*{0.06in}, \\
\mbox{\boldmath $\mathrm{q}$}_{4}(x)=\frac{1}{90}[x^{6}-360\mbox{\boldmath $\mathrm{q}$}_{2}(x)] \vspace*{0.06in},\\
\mbox{\boldmath $\mathrm{q}$}_{5}(x)=\frac{1}{336}[x^{7}-210\mbox{\boldmath $\mathrm{q}$}_{4}(x)]\vspace*{0.06in},\\
... 
\end{array} 
\]
On account of
\[ \small \left\{\begin{array}{ll} 
\begin{array}{rl}
\!\!\!\!\!\!D(x^{m+2})=\!\!\!\!&(m+2)(m+1)m(m-1)x^{m-2}+(m+2)(m+1)m x^{m-1} \\
		               &+(m+2)(m+1)(m-1)(m-3)x^{m} 
\end{array}\!\!\!\!\!\!,	& \mbox{if $m \not=1,3$}\\
\\
D(5x^{4}+x^{5}) = 120+240x, \!\!\!\!\!\!& \mbox{if $m=1$},\end{array}\right.                        
\]
$\mbox{\boldmath $\mathrm{q}$}$ is equally generated by the recurrence formula 
\[ \small \mbox{\boldmath $\mathrm{q}$}_{m}(x)\!\!=\!\!\left\{\begin{array}{ll} \begin{array}{rcl}\!\!\!\!\!\!\mbox{$\tiny{\frac{1}{(m+2)(m+1)(m-1)(m-3)}}$}[&\!\!\!\!\!\!x^{m+2}- \mbox{$\tiny{(m+2)(m+1)m}$}\mbox{\boldmath $\mathrm{q}$}_{m-1}(x) \\
&\!\!\!\!-\mbox{$\tiny{(m+2)(m+1)m(m-1)}$} \mbox{\boldmath $\mathrm{q}$}_{m-2}(x)]\end{array}\!\!\!, &\!\!\!\! \mbox{if $m \not=1,3$}\\
\\
\frac{1}{240}[5x^{4}+x^{5}-120\mbox{\boldmath $\mathrm{q}$}_{0}(x)], &\!\!\!\mbox{if $m=1$}\end{array}\right. 
\]
for all $m\in S^{\prime}$. It follows from (\ref{eq-ex2.ref}) that $1\in S^{\prime}$ and from (\ref{eq-ex2.bmr}) that $\deg\:  D(x^{n})\not=1$ for all $n\in \naturals$. Thus $\sigma: J \ni n\mapsto \sigma_{n}=\deg\:D(x^{n})\in S^{\prime}$ is not surjective.
Theorem \ref{theo-prim} implies that every complete sequence of CPs contains at least one derived-singular CP. In particular $\mbox{\boldmath $\mathrm{q}$}_{1}(x)$ is a derived-singular CP and so $\mbox{\boldmath $\mathrm{q}$}$ consists of primary-generic CPs for all $m\not=1,3$ and one derived-singular CP of index $m=1$. A basis of CPs is $\tilde{\mbox{\boldmath $\mathrm{q}$}}=\mbox{\boldmath $\mathrm{u}$}\cup \mbox{\boldmath $\mathrm{q}$}$. 

%Also the recurrence relation (\ref{eq-cre2}) generates the residual polynomials 
%\[ \small r_{m}(x)=\left\{\begin{array}{ll} \begin{array}{rcl}\!\!\!\!\!
%\mbox{$\tiny{\frac{1}{(m+2)(m+1)(m-1)(m-3)}}$}[\!\!\!\!\!&-&\!\!\!\!\! \mbox{$\tiny{(m+2)(m+1)m}$}r_{m-1}(x) \\
%&-&\!\!\!\!\!\mbox{$\tiny{(m+2)(m+1)m(m-1)}$} r_{m-2}(x)]\end{array}, & \!\!\!\mbox{if $m > 5$}\\
%\\
%r_{0}(x)=r_{1}(x)=r_{2}(x)=0,\;r_{4}(x)=\frac{1}{90}(120x^{3})=\frac{4}{3}x^{3},\vspace*{0.06in}\\ r_{5}(x)=\frac{1}{336}[840x^{3}-210r_{4}(x)]=\frac{5}{3}x^{3}.
%\end{array}\right. 
%\]
Let us consider the differential equation 
\begin{equation}\label{eq-ode.ex1}
	Dy(x)=g(x),
\end{equation}
where $D$ is given by (\ref{eq-DO2}) and $g(x)=\sum^{n}_{i=0}g_{i}x^{i}$ is a polynomial of arbitrary but fixed degree $n$. In view of (\ref{eq-can3}), the corresponding residual polynomials, given by $\mbox{\boldmath $\mathrm{r}$}_m(x)=D(\mbox{\boldmath $\mathrm{q}$}_m(x))-x^m$, are
\[
\mbox{\boldmath $\mathrm{r}$}_{0}(x)=\mbox{\boldmath $\mathrm{r}$}_{1}(x)=\mbox{\boldmath $\mathrm{r}$}_{2}(x)=0, \
\mbox{\boldmath $\mathrm{r}$}_{4}(x)=\frac{4}{3}x^{3},\
\mbox{\boldmath $\mathrm{r}$}_{5}(x)=\frac{5}{3}x^{3},\
\mbox{\boldmath $\mathrm{r}$}_{6}(x)=-\frac{10}{3}x^{3},
..., \mbox{\boldmath $\mathrm{r}$}_{n}(x)=r_{n3}\, x^3
\]
By (\ref{eq-parm}), $g(x)\in \mbox{$\mathrm{Im}$}(D)$ (STMC) if and only if 
\begin{equation}\label{eq-STMC}
g_{3}=\frac{4g_{4}}{3}+\frac{5g_{5}}{3}-\frac{10g_{6}}{3}+....+r_{n3}\,g_n
\end{equation}
Accordingly, if $g(x)$ is in the form
\begin{equation}\label{eq-rhs.ex1}
g(x)= g_{0}+g_{1}x+g_{2}x^{2} +(\frac{4g_{4}}{3}+\frac{5g_{5}}{3}-\frac{10g_{6}}{3}+...+r_{n3}\,g_n)x^{3}+g_{4}x^{4}+...+g_{n}x^{n},
\end{equation}
then the exact polynomial solution of (\ref{eq-ode.ex1}) is given by
\begin{equation}\label{eq-sol.ex1}
\begin{array}{ll}	y(x)= C_{0}+C_{1}x+C_{2}(x^{2}-x^{3})&\!\!\!\!+g_{0}\mbox{\boldmath $\mathrm{q}$}_{0}(x)+g_{1}\mbox{\boldmath $\mathrm{q}$}_{1}(x)+g_{2}\mbox{\boldmath $\mathrm{q}$}_{2}(x)\\                                                                   &\!\!\!\!+g_{4}\mbox{\boldmath $\mathrm{q}$}_{4}(x)+...+g_{n}\mbox{\boldmath $\mathrm{q}$}_{n}(x)
  \end{array} 
\end{equation}
Let $f(x)$ be a polynomial not necessarily in the polynomial range of $D$, defined in (\ref{eq-DO2}). Let also $Dy(x)=f(x)$ be the corresponding differential equation subjected by four supplementary conditions. As $h+\nu=-2+4=2$ and $N=5$, we need only two $\tau$-terms for any $n>5$ and so the perturbed equation must be in the form
\begin{equation}\label{eq-pert.ex1}
	Dy_{n}(x)=f(x)+\tau_{1}^{(n)}T_{n-1}(x)+\tau_{2}^{(n)}T_{n}(x).
\end{equation}
Let us rewrite the right hand side of (\ref{eq-pert.ex1}) as a single polynomial  of ascending powers, say $y(x)$. The  coefficients of $y(x)$ are linear combinations of $\tau_{1}^{(n)},\;\tau_{2}^{(n)}$. 
As $y(x)$ must be in  $\mbox{$\mathrm{Im}$}(D)$, we can identify $y(x)$ with $g(x)$. In particular, the STMC,  given by (\ref{eq-STMC}), must hold, providing a linear equation with unknowns $\tau_{1}^{(n)},\;\tau_{2}^{(n)}$. 
The four supplementary conditions plus the STMC would serve to determine the five parameters ($C_{0}, C_{1}, C_{2}, \tau_{1}^{(n)}, \tau_{2}^{(n)}$), which take place in the $\tau$-approximate solution.

\end{document}